%% file: risk_ctrl.tex
\documentclass{article}

\usepackage{fullpage}
\usepackage{natbib}
\usepackage[font=small,skip=0em, center]{caption}

\input{packages}

\input{shortcuts}
\input{notations}
\let\originalparagraph\paragraph
\renewcommand{\paragraph}[2][.]{\originalparagraph{#2#1}}

\def\short{0}

\title{On the Convergence of  the Iterative Linear \\  Exponential Quadratic Gaussian  Algorithm to Stationary Points}

\author{
		Vincent Roulet$^1$ \qquad 
		Maryam Fazel$^2$ \qquad
		Siddhartha Srinivasa$^3$ \qquad
		Zaid Harchaoui$^1$ \\[1ex]
		{\normalsize
		$^1$ Department of Statistics, University of Washington, Seattle} \\
		{\normalsize
		$^2$ Department of Electrical and Computer Engineering, University of Washington, Seattle} \\
		{\normalsize
		$^3$ Paul G. Allen School of Computer Science \& Engineering, University of Washington, Seattle} 
}
\date{}
\begin{document}
\maketitle

\begin{abstract}
A classical method for risk-sensitive nonlinear control is the iterative linear exponential quadratic Gaussian algorithm.
We present its convergence analysis from a first-order optimization viewpoint.
We identify the objective that the algorithm actually minimizes and we show how the addition of a proximal term guarantees convergence to a stationary point. 
\end{abstract}

\section*{Introduction}
\input{sections/01_intro}

\section{Risk-sensitive control}
\input{sections/02_pb_formulation}

\section{Iterative linearized risk-sensitive control}\label{sec:algos_cvg}
\input{sections/03_algos}

\input{sections/04_cvg}

\section{Numerical experiments}\label{sec:exp}
\input{sections/05_exp}

\section{Conclusion}
\input{sections/06_ccl}

\if\short0
\section*{Acknowledgements}
\input{sections/07_acknowledgments}
\fi

{\bibliography{bib_ctrl_dl.bib}
\bibliographystyle{abbrvnat}}

\if\short0
\onecolumn
\clearpage
\appendix

\section{Notations}\label{sec:notations}
\input{appendix/01_notations}

\section{Linear quadratic risk sensitive control}\label{sec:lin_quad_risk_ctrl}
\input{appendix/02_linear_quad_risk_ctrl}

\section{Iterative linearized algorithms}\label{sec:algos_app}
\input{appendix/03_algos_app}
\input{appendix/03b_algos_code}
\clearpage

\section{Convergence analysis proofs}\label{sec:cvg_proofs}
\input{appendix/04_cvg_proofs}

\section{Detailed experimental setting}\label{sec:detailed_exp}
\input{appendix/05_detailed_exp}

\fi

\end{document}

%% file: packages.tex
\usepackage{amsmath, amssymb, amsfonts}
\usepackage{amsthm, thmtools, thm-restate}
\usepackage{mathtools, bm, enumitem}

\usepackage{algorithm, algorithmicx, algpseudocode}

\usepackage{adjustbox, epsfig, graphicx}
\usepackage{nicefrac}
\usepackage{bbold}
\usepackage[skip=0em]{subcaption}

\usepackage{times}

\usepackage{color}

\usepackage{upgreek}

\newtheorem{theorem}{Theorem}[section]
\newtheorem{proposition}[theorem]{Proposition}
\newtheorem{definition}[theorem]{Definition}

\newtheorem{corollary}[theorem]{Corollary}

\newtheorem{remark}[theorem]{Remark}

\declaretheoremstyle[
notefont=\bfseries, notebraces={}{},
bodyfont=\normalfont\itshape,
headformat=\NAME \NOTE
]{nopar}

\definecolor{puorange}{rgb}{0.80,0.20,0}
\definecolor{bluegray}{rgb}{0.04,0,0.7}
\definecolor{greengray}{rgb}{0.05,0.50,0.15}
\definecolor{darkbrown}{rgb}{0.40,0.2,0.05}
\definecolor{darkcyan}{rgb}{0,0.4,1}
\definecolor{black}{rgb}{0,0,0}
\definecolor{grey}{rgb}{0.93,0.93,0.93}

\usepackage[colorlinks,citecolor=bluegray,linkcolor=darkbrown,urlcolor=blue,breaklinks]{hyperref}

%% file: shortcuts.tex
\newcommand{\reals}{{\mathbb R}}

\newcommand{\idm}{\operatorname{I}}
\newcommand{\id}{\operatorname{I}}

\newcommand{\Tr}{\operatorname{\bf Tr}}
\newcommand{\diag}{\operatorname{\bf diag}}

\newcommand{\Expect}{\operatorname{\mathbb E}}

\newcommand{\Prob}{\operatorname{\mathbb P}}

\DeclareMathOperator*{\argmax}{arg\,max}
\DeclareMathOperator*{\argmin}{arg\,min}

\def\shortdisplay{\setlength{\abovedisplayskip}{5pt}%
	\setlength{\belowdisplayskip}{5pt}%
	\setlength{\abovedisplayshortskip}{2pt}%
	\setlength{\belowdisplayshortskip}{2pt}}

\let\oldselectfont\selectfont
\def\selectfont{\oldselectfont\shortdisplay}

%% file: notations.tex
\newcommand{\horizon}{\tau}
\newcommand{\dyn}{\psi}
\newcommand{\dynexact}{\phi}

\newcommand{\traj}{\tilde x}

\newcommand{\stepsize}{\gamma}

\newcommand{\costfunc}{c}

\newcommand{\dist}{\operatorname{dist}}

%% file: sections/01_intro.tex
We present a convergence analysis of the classical iterative linear quadratic exponential Gaussian controller (ILEQG) \citep{whittle1981risk} for finite-horizon risk-sensitive or safe nonlinear control. The ILEQG algorithm is particularly popular in robotics applications~\citep{Li07} and can be seen as a risk-sensitive counterpart of the iterative linear quadratic Gaussian (ILQG) algorithm . We adopt here the viewpoint of the modern complexity analysis of first-order optimization algorithms as done by \citet{roulet2019iterative} for ILQG.

We address the following questions: (i) what is the convergence rate of ILEQG to a stationary point? (ii) how can we set the step-size to guarantee a decreasing objective along the iterations? The analysis we present here sheds light on these questions by highlighting the objective minimized by ILEQG which is a Gaussian approximation of a risk-sensitive cost around the linearized trajectory. We underscore the importance of the addition of a proximal regularization component for ILEQG to guarantee a worst-case convergence to a stationary point of the objective.  

The main result of the paper is Theorem~\ref{thm:conv}, where a sufficient decrease condition to choose the strength of the proximal regularization is given. The result also yields a complexity bound in terms of calls to a dynamic programming procedure implementable in a ``differentiable programming" framework, that is, a computational framework equipped with an automatic differentiation software library. We illustrate the variant of the iterative regularized linear quadratic exponential Gaussian controller we recommend on simple risk-sensitive nonlinear control examples.

\paragraph{Related work}
The linear exponential quadratic Gaussian algorithm is a fundamental algorithm for risk-sensitive or safe control~\citep{whittle1981risk, jacobson1973optimal, speyer1974optimization}. The algorithm builds upon a risk-sensitive measure, a less conservative and more flexible framework than the H$^\infty$ theory also used for robust control; see~\citep{glover1988state, hassibi1999indefinite, helton1999extending} and references therein.  
An excellent review of the classical results in abstract dynamic programming and control theory, in particular for risk-sensitive control, was done by~\citet{Bert18}.
Risk-measures were analyzed as instances of the optimized certainty equivalent applied to specific utility functions \citep{ben1986expected, ben2007old}. Risk-averse model predictive control was also studied to account for ambiguity in the knowledge of the underlying probability distribution~\citep{sopasakis2019risk}.

Algorithms for nonlinear control problems are usually derived by analogy to the linear case, which is solved in linear time with respect to the horizon by dynamic programming~\citep{Bell67}. In particular, the iterative linear quadratic regulator (ILQR) and iterative linear quadratic Gaussian (ILQG) algorithms are usually informally motivated as iterative linearization algorithms \citep{Li07}. A risk-sensitive variant with a straightforward optimization algorithm without theoretical guarantees was considered by~\citet{farshidian2015risk, ponton2016risk}.

On the first-order optimization front, optimization sub-problems such as Newton or Gauss-Newton-steps were shown to be implementable by using dynamic programming in classical works~\citep{Panto88, Dunn89, Side05}. Iterative linearized methods such as ILQR or ILQG were recently analyzed as Gauss-Newton-type algorithms and improved using proximal regularization and acceleration by extrapolation in~\citep{roulet2019iterative}. This work shares the same viewpoint and establishes worst-case complexity bounds for iterative linear quadratic exponential Gaussian controller (ILEQG) algorithms.

The companion code is available at {\small{\texttt{https://github.com/vroulet/ilqc}}}.
All proofs and notations are provided in 
\if\short1
the long version~\citep{roulet2019safeArxiv}.
\else
the Appendix.
\fi

%% file: sections/02_pb_formulation.tex
\paragraph{Problem formulation}
We consider discretized control problems stemming from continuous time settings with finite-horizon, see 
\if\short1
\citep{roulet2019safeArxiv}
\else
Appendix~\ref{sec:detailed_exp}
\fi
for the discretization step. Those are off-line control problems used for example at each step of a model predictive control framework.
We focus on the  control of a trajectory of length $\horizon$ composed of state variables $x_1, \ldots, x_\horizon \in \reals^d$ and controlled by parameters $u_0, \ldots, u_{\horizon-1} \in \reals^p$ through  dynamics $\dyn_t$ perturbed by i.i.d. white noise $w_t \sim \mathcal{N}(0, \sigma^2\id_q)$ such that
\begin{equation}\label{eq:noisy_non_lin_dyn_gen}
x_0 = \hat x_0,  \qquad x_{t+1} = \dyn_t(x_t, u_t, w_t),
\end{equation}
for $t=0,\ldots,\horizon-1$, where $\hat x_0$ is a fixed starting point and  the functions $\dyn_t: \reals^d \times \reals^p\times \reals^q\rightarrow \reals^d$ are assumed to be continuously differentiable. Precise assumptions for convergence are detailed in Sec.~\ref{sec:algos_cvg}.

Optimality is measured through convex costs $h_t$, $g_t$, on  the state and control variables $x_t$, $u_t$ respectively, defining the objective
\vspace{-1em}
\begin{equation}
	h(\bar x) + g(\bar u) = \sum_{t=1}^{\horizon} h_t(x_t) + \sum_{t=0}^{\horizon-1}g_t(u_t),	\label{eq:ctrl_obj}
\end{equation}
where $\bar x = (x_1;\ldots; x_\horizon) \in \reals^{\horizon d}$ is the trajectory, $\bar u = (u_0; \ldots; u_{\horizon-1}) \in \reals^{\horizon p}$ is the command, $h(\bar x) = \sum_{t=1}^{\horizon} h_t(x_t)$ and $g(\bar u) = \sum_{t=0}^{\horizon-1}g_t(u_t)$, and in the following we  denote by $\bar w = (w_0; \ldots; w_{\horizon-1}) \in \reals^{\horizon q}$ the noise.  For a given command $\bar u$, the dynamics in~\eqref{eq:noisy_non_lin_dyn_gen} define a probability distribution on the trajectories $\bar x$ that we denote $p(\bar x; \bar u)$.

The standard objective consists in minimizing the expected cost 
$
\min_{\bar u \in \reals^{\horizon p}} \:  \Expect_{\bar x \sim p(\cdot; \bar u)} \left[ h(\bar x)\right] + g(\bar u),
$
where $\bar x$ is a random variable following the model~\eqref{eq:noisy_non_lin_dyn_gen}.
We focus on risk-sensitive applications by minimizing
\begin{equation}\label{eq:risk_ctrl}
	\min_{\bar u \in \reals^{\horizon p}} \frac{1}{\theta}\log \Expect_{\bar x \sim p(\cdot; \bar u)} \big[ \exp \theta h(\bar x) \big] + g(\bar u), 
\end{equation}
for a given positive parameter $\theta>0$. If the dynamics are bounded, the risk-sensitive objective is well defined for any $\bar u$, otherwise it is only defined for small enough values of $\theta$ as illustrated in the linear quadratic case of Prop.~\ref{prop:lin_quad_risk_ctrl_as_min_max}. The risk-sensitive objective~\eqref{eq:risk_ctrl} seeks to minimize not only the expected objective but also higher moments as can be seen by expanding it around $\theta=0$,
\begin{align}\label{eq:taylor_risk}
\frac{1}{\theta}\log \Expect_{\bar x \sim p(\cdot; \bar u)} \left[ \exp \theta h(\bar x)\right]  & =  \Expect_{\bar x \sim p(\cdot; \bar u)}\left[h(\bar x)\right]  +  \frac{\theta}{2} \operatorname{\mathbb{V}ar}_{\bar x \sim p(\cdot; \bar u)} \left[h(\bar x) \right] + \mathcal{O}(\theta^2),
\end{align}
which also shows that for $\theta\rightarrow 0$ we retrieve the expected cost. In Fig.~\ref{fig:risk_sensitivie_illus} we illustrate the smoothness effect of the risk-sensitive objective, which, for larger values of $\theta$, tends to select the most stable minimizers, i.e., the ones with the largest valley, see~\citep{dvijotham2014universal} for a detailed discussion. An application of the risk-sensitive cost is to make the controller robust to a random disturbance noise that would affect the dynamics at a given time (like a kick on the machine). Although the risk-sensitive controller may not pick the minimal cost of the original function, we can expect the risk-sensitive controller to be robust against disturbance noise as illustrated in Fig.~\ref{fig:expected_behavior}.
\begin{figure}[t]
	\centering
	\begin{minipage}{0.5\linewidth}
		\centering
		\includegraphics[width=0.7\linewidth]{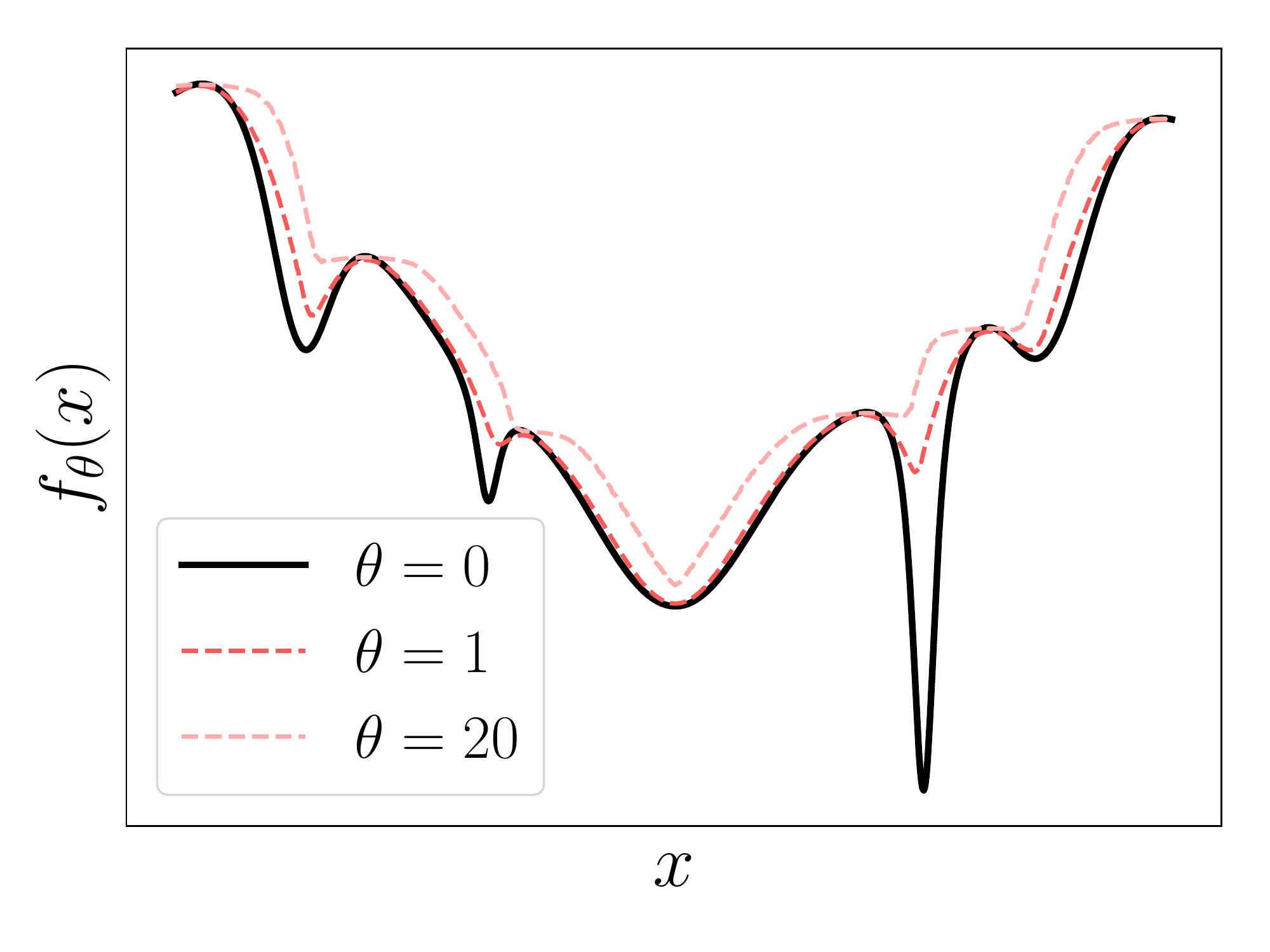}
		\captionof{figure}{Effect of the risk-sensitive parameter $\theta$ 
			for $f_\theta(x){ = }\frac{1}{\theta}\log \Expect_{w\sim\mathcal{N}(0,1)} \big[ \exp \theta F(x{+}w) \big]$   \\
			with $F$ illustrated by the black line.\label{fig:risk_sensitivie_illus}}
	\end{minipage}~
	\begin{minipage}{0.6\linewidth}
		\centering
		\includegraphics[width=0.6\linewidth]{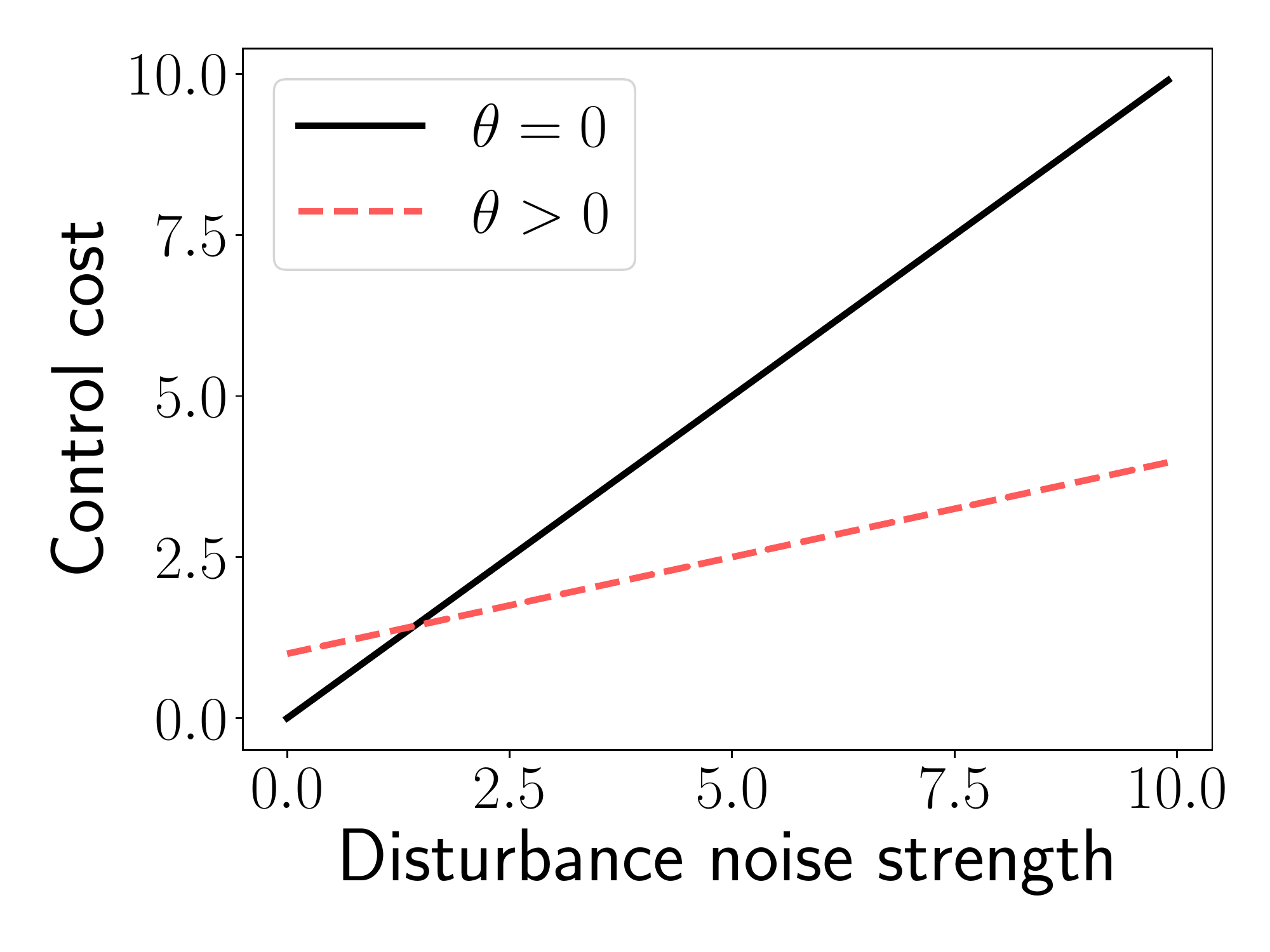}
		\captionof{figure}{Expected behavior of the risk-sensitive \\ controllers  for increasing disturbance noise.\label{fig:expected_behavior}}
	\end{minipage}
\end{figure}

\paragraph{Linear Quadratic Exponential Gaussian control}
The resolution of non-linear risk-sensitive control problems rest on the linear quadratic case whose properties are recalled below. 
\begin{restatable}{proposition}{linquadrisk}\label{prop:lin_quad_risk_ctrl_as_min_max}
	Consider quadratic objectives and linear dynamics defined by
	\begin{equation}
	h_t(x_t) = \frac{1}{2}x_t^\top H_t x_t + \tilde h_t^\top x_t, \quad g_t(u_t) = \frac{1}{2}u_t^\top G_t u_t + \tilde g_t^\top u_t, \quad x_{t+1} = A_t x_t + B_t u_t + C_t w_t, \label{eq:lin_quad}
	\end{equation}
	where $H_t \succeq 0$, $G_t \succ 0$, $w_t \sim \mathcal{N}(0, \sigma^2\id_q)$.
	 and denote by $H, \tilde B, \tilde C, \tilde x_0$ the matrices and vector such that for any trajectory $\bar x$, $H = \nabla^2 h(\bar x)$, $\bar x = \tilde B \bar u + \tilde C \bar w +\tilde x_0$. 
	 We have that
	\begin{enumerate}[label=(\roman*), nosep] 
		\item the risk sensitive control problem~\eqref{eq:risk_ctrl} is equivalent to\footnote{By equivalent, we mean that the two problems share the same set of minimizers.}
	\end{enumerate}
	\begin{align}
	\min_{\bar u \in \reals^{\horizon p}} \sup_{\bar w \in \reals^{\horizon q}} Q(\bar u, \bar w) = \min_{\bar u \in \reals^{\horizon p}}  \sup_{\substack{\bar w \in \reals^{\horizon q}\\ \bar x \in \reals^{\horizon d}}} \: & \sum_{t=1}^{\horizon}\frac{1}{2}x_t^\top H_t x_t + \tilde h_t^\top x_t + \sum_{t=0}^{\horizon-1} \frac{1}{2}u_t^\top G_t u_t + \tilde g_t^\top u_t - \sum_{t=0}^{\horizon-1} \frac{1}{2\theta\sigma^2}\|w_t\|_2^2 \label{eq:min_max_ctrl} \\
	\mbox{\textup{subject to}} \quad & x_{t+1} = A_tx_t + B_tu_t + C_tw_t \nonumber\\
	& x_0 = \hat x_0, \nonumber
	\end{align}
	where $Q$ is a quadratic in $\bar u, \bar w$ obtained from the right hand side by expressing $\bar x$ in terms of $\bar u, \bar w$,
	\begin{enumerate}
		\item[(ii)] if $(\theta\sigma^2)^{-1} < \lambda_{\max}(\tilde C^\top H \tilde C)$ the quadratic $Q$ is not concave in $\bar w$ such that the risk-sensitive objective is not defined,
		\item[(iii)] if $(\theta\sigma^2)^{-1} > \lambda_{\max}(\tilde C^\top H \tilde C)$,  the quadratic $Q$ is strongly concave in $\bar w$ and the risk-sensitive problem can be solved analytically by dynamic programming.
	\end{enumerate}
\end{restatable}
The resolution of the control problem by dynamic programming checks if the quadratic defining the objective is concave in $\bar w$ during the backward pass, otherwise the problem is not defined. 
Each cost-to-go function is indeed a quadratic whose positive-definiteness determines the feasibility of the problem. The detailed implementation is provided in 
\if\short1
~\citep{roulet2019safeArxiv}.
\else
Appendix~\ref{sec:lin_quad_risk_ctrl}.
\fi

\paragraph{Iterative Linearized Quadratic Exponential Gaussian}
A common method to tackle the non-linear risk-sensitive control problem is the Iterative Linearized Quadratic Exponential Gaussian (ILEQG) algorithm, that (i) linearizes the dynamics and approximates quadratically the objectives around the current command and associated noiseless trajectory, (ii) solves the associated linear quadratic problem to get an update direction, (iii) moves along the update direction using a line-search. 

{Formally, at a given command $\bar u^{(k)}$ with associated noiseless trajectory $\bar x^{(k)}$ given by $x_0^{(k)}{ = }\hat x_0$, $x_{t+1}^{(k)}{ =} \dyn_t(x_t^{(k)}, u_t^{(k)},0)$, am update direction is given by the solution $\bar v^*$, if it exists, of}
\begin{align}\label{eq:ileqg}
\min_{\bar v \in \reals^{\horizon p}} \sup_{\substack{\bar w \in \reals^{\horizon p} \\ \bar y \in \reals^{\horizon d}}} \: & \sum_{t=1}^{\horizon}\left(\frac{1}{2}y_t^\top H_t y_t + \tilde h_t^\top y_t\right) + \sum_{t=0}^{\horizon-1} \left(\frac{1}{2}v_t^\top G_t v_t + \tilde g_t^\top v_t\right)  - \sum_{t=0}^{\horizon-1} \frac{1}{2\theta \sigma^2}\|w_t\|_2^2 \\
\mbox{\textup{subject to}} \quad & y_{t+1} = A_t y_t +B_tv_t + C_tw_t \nonumber\\
& y_0 = 0, \nonumber
\end{align} 
where 
\begin{gather*}
	A_t = \nabla_x \dyn_t(x_t^{(k)}, u_t^{(k)}, 0)^\top \quad B_t = \nabla_u \dyn_t(x_t^{(k)}, u_t^{(k)}, 0)^\top \quad  C_t = \nabla_w \dyn_t(x_t^{(k)}, u_t^{(k)}, 0)^\top  \\
	H_t = \nabla^2 h_t(x_t^{(k)})\quad \tilde h_t = \nabla h_t(x_t^{(k)}) \quad G_t =\nabla^2g_t(u_t^{(k)})\quad  \tilde g_t = \nabla g_t(u_t^{(k)}).
\end{gather*}
The next command is given by 
\[
\bar u^{(k+1)} = \bar u^{(k)} + \stepsize \bar v^*,
\]
where $\stepsize$ is a step-size chosen by line-search. The complete pseudo-code is presented in 
\if\short1
~\citep{roulet2019safeArxiv}.
\else
Appendix~\ref{sec:algos_app}.
\fi
The objective of this work is to understand the relevance of this method and to improve its implementation by answering the following questions: 
\begin{enumerate}[nosep]
	\item Does ILEQG ensure the decrease of the risk-sensitive objective? If yes, what is its rate of convergence?
	\item How can the step-size be chosen to ensure the monotonicity of the algorithm in a principled way?
\end{enumerate}

%% file: sections/03_algos.tex
\subsection{Model minimization}
We analyze the ILEQG method as a model-minimization scheme. To ease the exposition, we consider the case of additive noise, i.e., dynamics of the form, 
\begin{equation}\label{eq:noisy_non_lin_dyn}
x_0 = \hat x_0,  \qquad x_{t+1} = \dynexact_t(x_t, u_t+  w_t).
\end{equation}
for bounded continuously  differentiable dynamics~$\dynexact_t: \reals^d \times \reals^p \rightarrow \reals^d$. Note that it implies $p=q$ in the previous framework. The algorithm and its interpretation can be extended to the general case~\eqref{eq:noisy_non_lin_dyn_gen}, see Appendix~\ref{sec:algos_app} and~\ref{sec:cvg_proofs}.

First, we consider the noiseless trajectory as a function $\traj : \reals^{\horizon p} \rightarrow \reals^{\horizon d}$ of the control variables, decomposed as $\traj(\bar u) = (\traj_1(\bar u); \ldots; \traj_\horizon(\bar u))$ where
\begin{equation}\label{eq:traj_def}
\traj_1(\bar u) = \dynexact_0(\hat x_0, u_0),  \quad \traj_{t+1}(\bar u) = \dynexact_t(\traj_t(\bar u), u_t),
\end{equation}
such that the noisy trajectory is given by $\traj(\bar u +\bar w)$. The risk sensitive objective~\eqref{eq:risk_ctrl} can then be written as
\begin{equation}
\min_{\bar u \in \reals^{\horizon p}} f_\theta(\bar u) = \eta_{\theta}(\bar u) 
+ g(\bar u), \label{eq:non_lin_risk_ctrl} \qquad
 \mbox{with} \qquad  \eta_{\theta}(\bar u)  = \frac{1}{\theta}\log \Expect_{\bar w} \Big[ \exp \theta h\big(\traj(\bar u + \bar w)\big)\Big],
\end{equation}
where, here and thereafter, $\bar w \sim \mathcal{N}(0,\sigma^2\id_{\horizon p})$ unless specified differently.
Now, at a current command $\bar u$, for a given control deviation $\bar v$,  the random trajectory $\traj(\bar u + \bar v +  \bar w)$ is approximated as a perturbed trajectory of $\traj(\bar u)$, by
\begin{align}
\traj(\bar u + \bar v + \bar w) & \approx \traj(\bar u) + \nabla \traj(\bar u)^\top (\bar v +  \bar w) \label{eq:approx_traj}.
\end{align} 
The objective is then approximated as $f_\theta(\bar u + \bar v) \approx m_{f_\theta}(\bar u + \bar v ; \bar u) $, where 
\begin{align}
m_{f_\theta}(\bar u{ + }\bar v ; \bar u) &{\triangleq }\frac{1}{\theta}\log \Expect_{\bar w }  \exp \theta q_h\big(\bar x{ + }\nabla \traj(\bar u)^\top\bar v{ + }\nabla\traj(\bar u)^\top \bar w  ; \bar x  \big) + q_g(\bar u + \bar v; \bar u), \label{eq:model}
\end{align} 
$q_h(\bar x + \bar y; \bar x) \triangleq h(\bar x) + \nabla h(\bar x)^\top \bar y + \bar y^\top\nabla^2 h(\bar x) \bar y/2$, $q_g(\bar u + \bar v; \bar u)$ is defined similarly and $\bar x = \traj(\bar u)$ is the noiseless trajectory. As the following proposition clarifies, the update direction computed by ILEQG in~\eqref{eq:ileqg} is given by minimizing directly the model $m_{f_\theta}$. Yet, from an optimization viewpoint, a regularization term must be added to this minimization to ensure that the solutions stay in a region where the model is valid.
Formally, we consider a regularized variant of ILEQG, we call RegILEQG, that starts at a point $\bar u^{(0)}$ and  defines the next iterate as
\begin{equation}\label{eq:min_model_step}
\bar u^{(k+1)} {=} \bar u^{(k)} { + }\argmin_{\bar v \in \reals^{\horizon p}}\left\{ m_{f_\theta}(\bar u^{(k)} {+} \bar v; \bar u^{(k)}){ +} \frac{1}{2\gamma_k} \|\bar v\|_2^2\right\},
\end{equation}
where $ \gamma_k$ is the step-size: the smaller $\gamma_k$ is, the closer the solution is to the current iterate. 
The following proposition shows that the minimization step~\eqref{eq:min_model_step} amounts to a linear quadratic exponential Gaussian risk-sensitive control problem.

\begin{proposition}\label{prop:model_min}
	The model minimization step~\eqref{eq:min_model_step} is given as $\bar u^{(k+1)} = \bar u^{(k)} + \bar v^*$ where $\bar v^*$ is the solution, if it exists, of
	\begin{align}
	\min_{\bar v \in \reals^{\horizon p}} \sup_{\substack{\bar w \in \reals^{\horizon p}\\ \bar y \in \reals^{\horizon d}}} \quad & \sum_{t=1}^{\horizon}\left(\frac{1}{2}y_t^\top H_t y_t + \tilde h_t^\top y_t\right)  + \sum_{t=0}^{\horizon-1} \left(\frac{1}{2}v_t^\top (G_t + \gamma_k^{-1}\id_p)v_t + \tilde g_t^\top v_t\right) \label{eq:model_min_lin_quad_risk_ctrl}- \sum_{t=0}^{\horizon-1} \frac{1}{2\theta \sigma^2}\|w_t\|_2^2 \\
	\mbox{\textup{subject to}} \quad & y_{t+1} = A_t y_t +B_t v_t + B_t w_t  \nonumber\\
	& y_0 = 0, \nonumber
	\end{align}
	where, denoting $x_t^{(k)} = \traj_t(\bar u^{(k)})$,
	\begin{gather*}
		A_t = \nabla_x \dynexact_t(x_t^{(k)}, u_t^{(k)})^\top \quad B_t =  \nabla_u \dynexact_t(x_t^{(k)}, u_t^{(k)})^\top \\
		H_t = \nabla^2 h_t(x_t^{(k)}) \quad \tilde h_t = \nabla h_t(x_t^{(k)}) \quad
		G_t = \nabla^2g_t(u_t^{(k)}) \quad \tilde g_t = \nabla g_t(u_t^{(k)}). 
	\end{gather*}  
\end{proposition}
Each model-minimization step can then be performed by dynamic programming. The overall algorithm for general dynamics of the form~\eqref{eq:noisy_non_lin_dyn_gen} is presented in 
\if\short1
~\citep{roulet2019safeArxiv}.
\else
Appendix~\ref{sec:algos_app}.
\fi Note that for simplified dynamics~\eqref{eq:noisy_non_lin_dyn}, the matrix $C_t$ defined in~\eqref{eq:ileqg} reduces to $B_t$. As detailed in~Appendix~\ref{sec:algos_app}, ILEQG is indeed an instance of RegILEQG with infinite step-size. 
If the costs depend only on the final state, i.e.,  $h(\bar x) = h_\horizon(x_\horizon)$, the steps can be computed more efficiently by making calls to automatic differentiation oracles, see 
\if\short1
~\citep{roulet2019safeArxiv}
\else
Appendix~\ref{sec:algos_app}
\fi 
for more details.

%% file: sections/04_cvg.tex
\subsection{Convergence analysis}
We analyze the behavior of the regularized variant of ILEQG for quadratic convex costs $h_t$, $g_t$, a common setting in applications. Our main contribution is to show that the algorithm can be seen to minimize a surrogate of the risk-sensitive cost. The algorithm can indeed be decomposed in two different approximations:
\begin{enumerate}[label=(\roman*), nosep]
	\item the random trajectories are approximated by Gaussians defined by the linearization of the dynamics,
	\item the non-linear control of the trajectory is approximated by a linear control defined by the linearization of the dynamics.
\end{enumerate}
We show that the first approximation makes the algorithm work on a surrogate of the true risk-sensitive objective. By identifying this surrogate, we can improve the implementation of the algorithm.

\paragraph{Surrogate risk-sensitive cost}
By approximating the noisy trajectory by a Gaussian variable using first-order information of the trajectory, we define the surrogate risk-sensitive objective as follows
\begin{equation}
 \hat f_\theta(\bar u) = \hat \eta_\theta(\bar u) + g(\bar u), \quad 
\mbox{with} \quad    \hat \eta_\theta(\bar u)  = \frac{1}{\theta}\log\Expect_{\bar w} \exp[\theta h(\traj(\bar u)  +\nabla \traj(\bar u)^\top \bar w)].  \label{eq:approx_risk}
\end{equation}
The surrogate risk-sensitive objective is essentially the log-partition function of a Gaussian distribution defined by the linearized trajectory  as shown in the following proposition. 
\clearpage
\begin{restatable}{proposition}{approxrisk}\label{prop:risk_approx_comput}
	For $\bar u \in \reals^{\horizon p}$ with $\bar x = \tilde x(\bar u)$, if 
	\begin{equation}\label{eq:approx_risk_cond}
	\sigma^{-2} \id_{\horizon p}\succ \theta \nabla \tilde x(\bar u) \nabla^2h(\bar x)  \nabla \tilde x(\bar u)^\top,
	\end{equation}
	the surrogate $\hat \eta_\theta$ in~\eqref{eq:approx_risk} is well-defined and is the scaled log-partition function of 
	\begin{align}\label{eq:gaussian_prob_def}
	\hat p(\bar w; \bar u) & = \exp\left(\theta h(\traj(\bar u){  +}\nabla \traj(\bar u)^\top \bar w) {-}\frac{1}{2\sigma^2} \|\bar w\|_2^2{ - }\theta \hat \eta_\theta(\bar u)\right),
	\end{align}
	which is the density of a Gaussian $\mathcal{N}(\bar w_*, \Sigma)$ with 
	\begin{align}\label{eq:gaussian_def}
	\bar w_*  = \theta \Sigma X\tilde h,   \qquad
	\Sigma  = (\sigma^{-2}\id_{\horizon p} -\theta XHX^\top)^{-1},
	\end{align}
	where  $X = \nabla \traj(\bar u)$, $ \tilde h = \nabla h(\bar x) , H = \nabla^2 h (\bar x) $ and $\bar x  = \traj(\bar u)$. 
	Therefore, the surrogate risk-sensitive objective can be computed analytically.
\end{restatable}
The approximation error induced by using the surrogate instead of the original risk-sensitive cost is illustrated in Sec.~\ref{sec:exp}. Note that the surrogate $\hat \eta_\theta(\bar u)$  in~\eqref{eq:approx_risk} shares similar properties as the original cost in~\eqref{eq:taylor_risk}, since it can be extended around $\theta =0$ to 
\begin{align*}
\hat\eta_\theta(\bar u) = \:& h(\tilde x(\bar u ))  + \Expect_{\bar w \sim \hat p(\cdot; \bar u)}\bar w^\top \nabla \tilde x(\bar u) \nabla^2 h(\tilde x (\bar u)) \nabla \tilde x(\bar u)^\top \bar w  + \frac{\theta}{2}\operatorname{\mathbb{V}ar}_{\bar w \sim \hat p(\cdot; \bar u)} h(\tilde x(\bar u) +\nabla \tilde x(\bar u)^\top \bar w)) + \mathcal{O}(\theta^2).
\end{align*}
Namely, it accounts not only for the cost of the noiseless trajectory but also for the variance defined by the linearized trajectories. 
Provided that condition~\eqref{eq:approx_risk_cond} holds, the gradient of the surrogate risk-sensitive cost reads (see 
\if\short1
\citep{roulet2019safeArxiv})
\else
Appendix~\ref{sec:cvg_proofs})
\fi
\begin{align*}
\nabla \hat \eta_\theta(\bar u)  
& = \Expect_{\bar w \sim \hat p(\cdot;\bar u)}(\nabla \traj(\bar u){ + }\nabla^2 \traj(\bar u)[\cdot, \bar w, \cdot])
\nabla h(\traj( \bar u){ + }\nabla \traj (\bar u)^\top \bar w),
\end{align*}
where $\hat p(\cdot;\bar u)$ is defined in~\eqref{eq:gaussian_prob_def}.
The analysis of the algorithm requires to define also the truncated gradient of the surrogate risk-sensitive cost as
\begin{align*}
	\widehat \nabla \hat \eta_\theta(\bar u) 
	& =  \Expect_{\bar w \sim \hat p(\cdot; \bar u)} \nabla \traj(\bar u)
	\nabla h(\traj( \bar u) + \nabla \traj (\bar u)^\top w).
\end{align*}
We link the model-minimization steps of the regularized variant of ILEQG to the truncated gradient in the following proposition.
\begin{restatable}{proposition}{RegILEQGsteptruncgrad}\label{prop:ILEQG_trunc_grad}
	Consider the regularized iterative linear exponential Gaussian iteration~\eqref{eq:min_model_step}, if condition~\eqref{eq:approx_risk_cond} holds on $\bar u^{(k)}$, the model $m_{f_\theta}$ in~\eqref{eq:model} is well-defined and convex and the step reads
\begin{align*}
\bar u^{(k+1)} = \bar u^{(k)}- & (G+\gamma_k^{-1}\id_{\horizon p}+XHX^\top{+ } \theta V )^{-1}(\nabla g(\bar u^{(k)}) + \widehat \nabla \hat \eta_\theta(\bar u^{(k)})),
\end{align*}
where 
\begin{align*}
V & = \operatorname{\mathbb{V}ar}_{\bar w \sim \hat p(\cdot; \bar u^{(k)})} \nabla \traj(\bar u^{(k)})
\nabla h(\traj( \bar u^{(k)}) + \nabla \traj (\bar u^{(k)})^\top w)   = XHX^\top(\sigma^{-2} \id_{\horizon p} -\theta XHX^\top)^{-1}XHX^\top
\end{align*}
and $X{ = }\nabla \traj(\bar u^{(k)})$, $H{ = }\nabla^2 h (\bar x)$,  $G{ = }\nabla^2g (\bar u^{(k)})$, $\bar x {= }\traj(\bar u^{(k)})$.
\end{restatable}

\paragraph{Convergence to stationary points}
We make the following assumptions for our analysis.
\begin{restatable}{assumption}{asscvg}\label{ass:cvg}~
	\begin{enumerate}[nosep]
		\item The dynamics $\dynexact_t$ are twice differentiable, bounded, Lipschitz, smooth such that the trajectory function $\traj$ is also twice differentiable, bounded, Lipschitz and smooth. Denote by $\ell_{\traj}$ and $L_{\traj}$ the Lipschitz continuity and smoothness constants respectively of $\traj$ and define $M_{\traj} = \max_{\bar u \in \reals^{\horizon p}} \dist(\traj(\bar u) , X^*)$, where $X^* = \argmin_{\bar x \in \reals^{\horizon d}}h(\bar x)$.
		\item The costs $h$ and $g$ are convex quadratics with smoothness constants $L_h,L_g$.
		\item The risk-sensitivity parameter is chosen such that $\tilde \sigma^{-2} = \sigma^{-2}-\theta L_h \ell_{\traj}^2 >0$, which  ensures  that condition~\eqref{eq:approx_risk_cond} holds for any $\bar u \in \reals^{\horizon p}$.
	\end{enumerate}
\end{restatable}
The following proposition shows stationary convergence for the regularized variant of ILEQG as an optimization method of the surrogate risk-sensitive cost. The additional constant term is due to the truncation of the gradient of the surrogate risk-sensitive cost.
\begin{restatable}{theorem}{cvg}\label{thm:conv}
Under Asm.~\ref{ass:cvg}, suppose that the step-sizes of the regularized iterative linear exponential Gaussian iteration~\eqref{eq:min_model_step} are chosen such that
\begin{equation}\label{eq:suff_decrease}
\hat f_\theta(\bar u^{(k+1)}) \leq m_{f_\theta}(\bar u^{(k+1)}; \bar u^{(k)}) + \frac{1}{2\gamma_k}\|\bar u^{(k+1)} - \bar u^{(k)}\|_2^2,
\end{equation}
with $\gamma_k \in[\gamma_{\min}, \gamma_{\max}]$.
Then, the surrogate objective $\hat f_\theta$ decreases and after $K$ iterations we have 
\[
\min_{k =0,\ldots, K-1}\|\nabla \hat f_\theta(\bar u^{(k)})\|_2 \leq L\sqrt{\frac{2 (\hat f_\theta(\bar u^{(0)}) - \hat f_\theta(\bar u^{(K)}))}{K}}+\delta,
\]
where \begin{align*}
	L & = \max_{\gamma \in [\gamma_{\min}, \gamma_{\max}]} \sqrt{\gamma}(L_g + \gamma^{-1} + (\tilde \sigma/\sigma)^{2}\ell_{\traj}^2L_h) \\
	\delta & = \theta \tilde \sigma^2 L_h^2  L_{\tilde x}   \ell_{\traj} M_{\traj}^2+ \theta^2 \tilde \sigma^4 L_h^3L_{\traj}\ell_{\traj}^3 M_{\traj}^2 + \tau p \tilde \sigma^2 L_h L_{\traj}  \ell_{\traj}.
\end{align*}
\end{restatable}
Previous proposition gives a criterion~\eqref{eq:suff_decrease} for line-searches. We show in 
\if\short1
\citep{roulet2019safeArxiv}
\else
Appendix~\ref{sec:cvg_proofs}
\fi 
that there exists a step-size $\hat \gamma$ such that condition~\eqref{eq:suff_decrease} is satisfied along the iterations. With this step-size, the number of steps to get an $\epsilon + \delta$ stationary point  is at most
\[
\frac{2\hat \gamma(L_g + \hat \gamma^{-1} + (\tilde \sigma/\sigma)^{2}\ell_{\traj}^2L_h)^2 (\hat f_\theta(\bar u^{(0)}) - \hat f^*_\theta)}{\epsilon^2}.
\]

%% file: sections/05_exp.tex
\subsection{Experimental setting}
Detailed description of the parameters setting can be found in
\if\short1
\citep{roulet2019safeArxiv}.
\else
Appendix~\ref{sec:detailed_exp}.
\fi
\paragraph{Control settings}
We apply the risk-sensitive framework to two classical continuous time control settings: swinging-up a pendulum and moving a two-link arm robot, both detailed in 
\if\short1
\citep{roulet2019safeArxiv}.
\else
Appendix~\ref{sec:detailed_exp}.
\fi
Their discretization  leads to dynamics of the form
\begin{equation}\label{eq:discrete_dyn}
\begin{split}
x_{1,t+1} & = x_{1,t} + \delta x_{2,t} \\
x_{2,t+1} & = x_{2,t} + \delta f(x_{1,t}, x_{2,t}, u_t), 
\end{split}
\end{equation}
for $t = 0,\ldots \horizon-1$, where $x_1, x_2$ describe the position and the speed of the system respectively, $f$ defines the dynamics derived by Newton's law, $\delta$ is the time step, $u$ is a force that controls the system.

\paragraph{Noise modeling}
The risk-sensitive cost is defined by an additional noisy force applied to the dynamics. Formally, the discretized dynamics~\eqref{eq:discrete_dyn} are modified as
\begin{equation}\label{eq:risk_dyn}
\begin{split}
x_{1,t+1} & = x_{1,t} + \delta x_{2,t}  \\
x_{2,t+1} & = x_{2,t} + \delta f(x_{1,t}, x_{2,t}, u_t + w_{t}), 
\end{split}
\end{equation} 
for $t = 0,\ldots ,\horizon-1$, where $ w_t \sim \mathcal{N}(0, \sigma^2\id_p)$ and $\sigma$ is chosen to avoid chaotic behavior, see 
\if\short1
\citep{roulet2019safeArxiv}.
\else
Appendix~\ref{sec:detailed_exp}.
\fi

We test the optimized expected or risk-sensitive costs on a setting where the dynamics are perturbed at a given time $t_w$ by a force of amplitude $\rho$. This models the robustness of the control against kicking the robot. Formally,  we analyze the performance of the solutions of the expected cost (denoted $\theta=0$) or the risk-sensitive cost~\eqref{eq:risk_ctrl} on dynamics of the form
\begin{equation}\nonumber
\begin{split}
x_{1,t+1} & = x_{1,t} + \delta x_{2,t}  \\
x_{2,t+1} & = x_{2,t} + \delta f(x_{1,t}, x_{2,t}, u_t + \rho \mathbb{1}(t = t_w)), 
\end{split}
\end{equation} 
for $t = 0,\ldots ,\horizon-1$, where $\rho \sim \mathcal{N}(0, \sigma_{test} \id_p)$
with the same cost $h(\bar x)$ computed as an average on $n=100$ simulations. We call this cost the test cost.

\subsection{Results}
\begin{figure}[t]
	\begin{center}
		\includegraphics[width=0.32\linewidth]{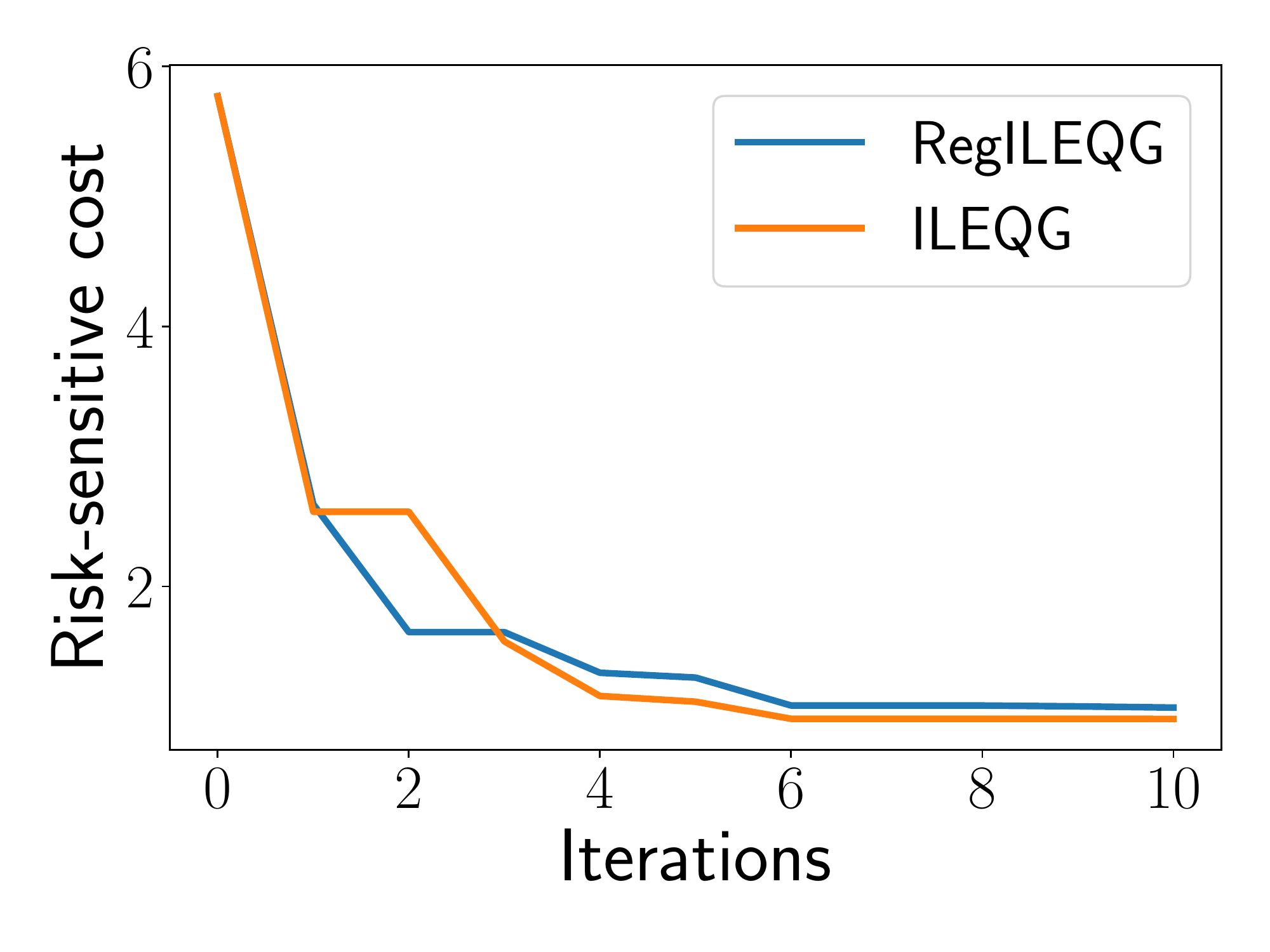}~
		\includegraphics[width=0.32\linewidth]{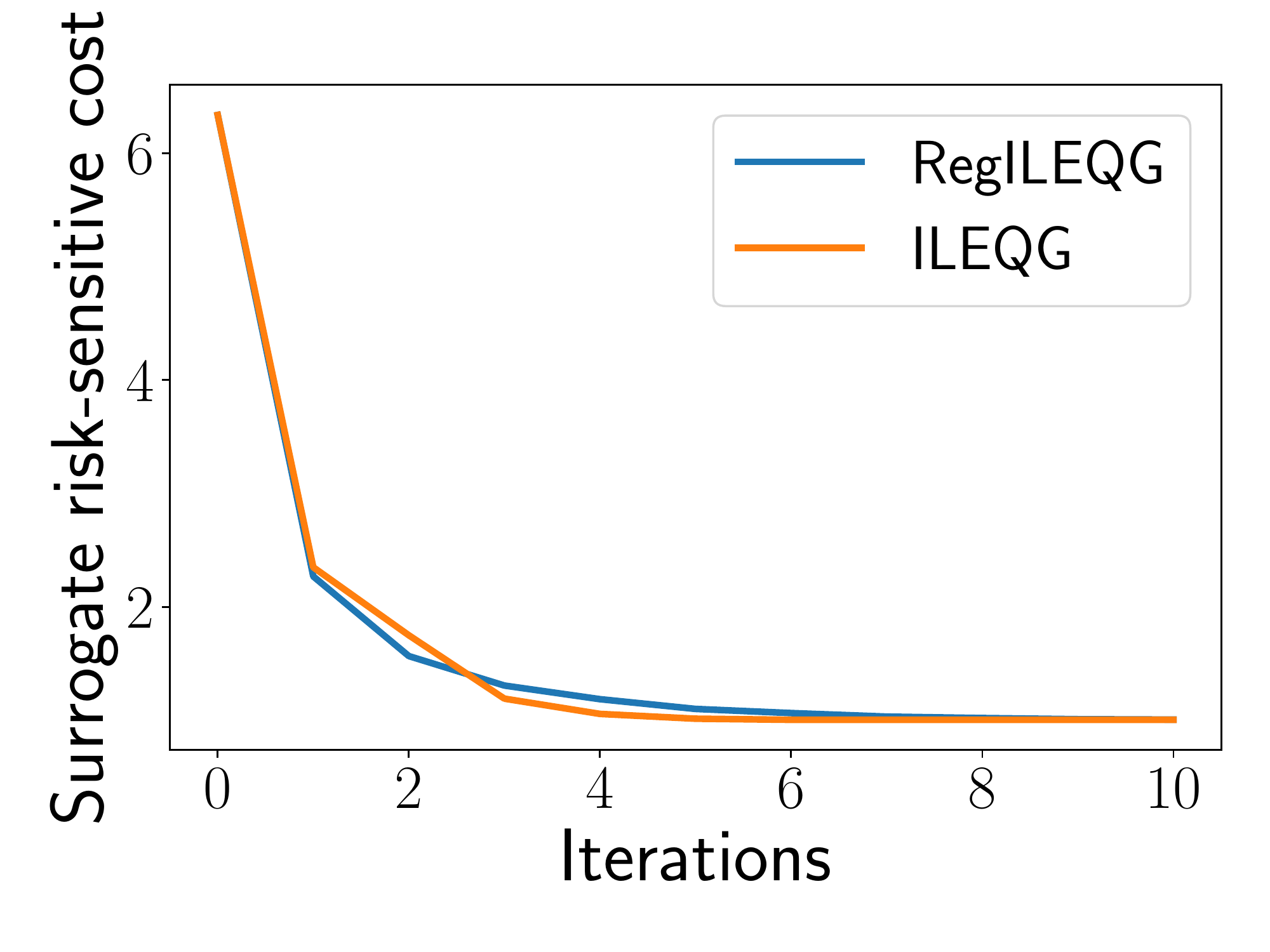}
		\caption{
			Convergence of iterative linearized methods, \\ RegILEQG and ILEQG,  on the pendulum problem. \label{fig:conv}
		}
		\includegraphics[width=0.32\linewidth]{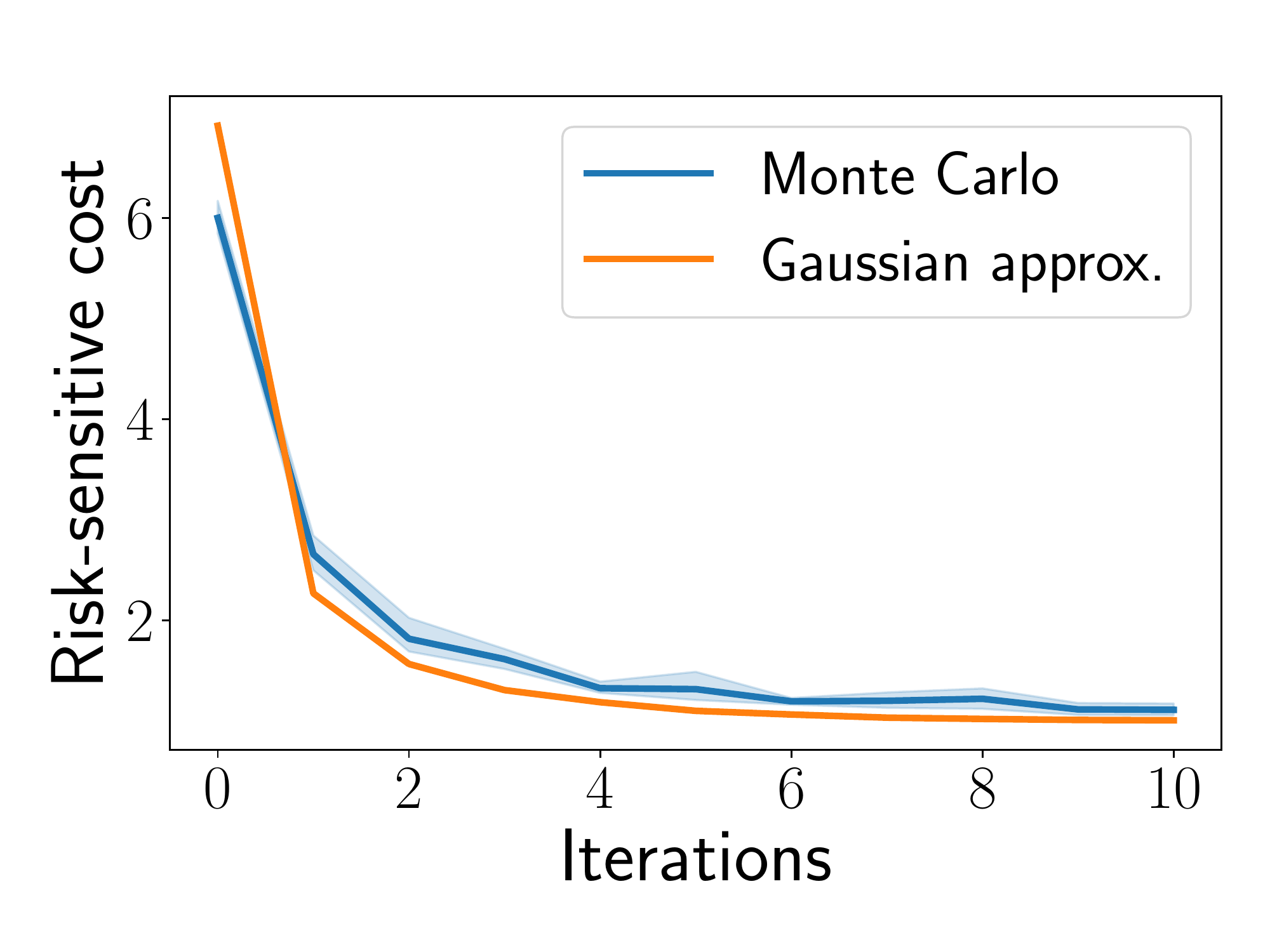}~
		\includegraphics[width=0.32\linewidth]{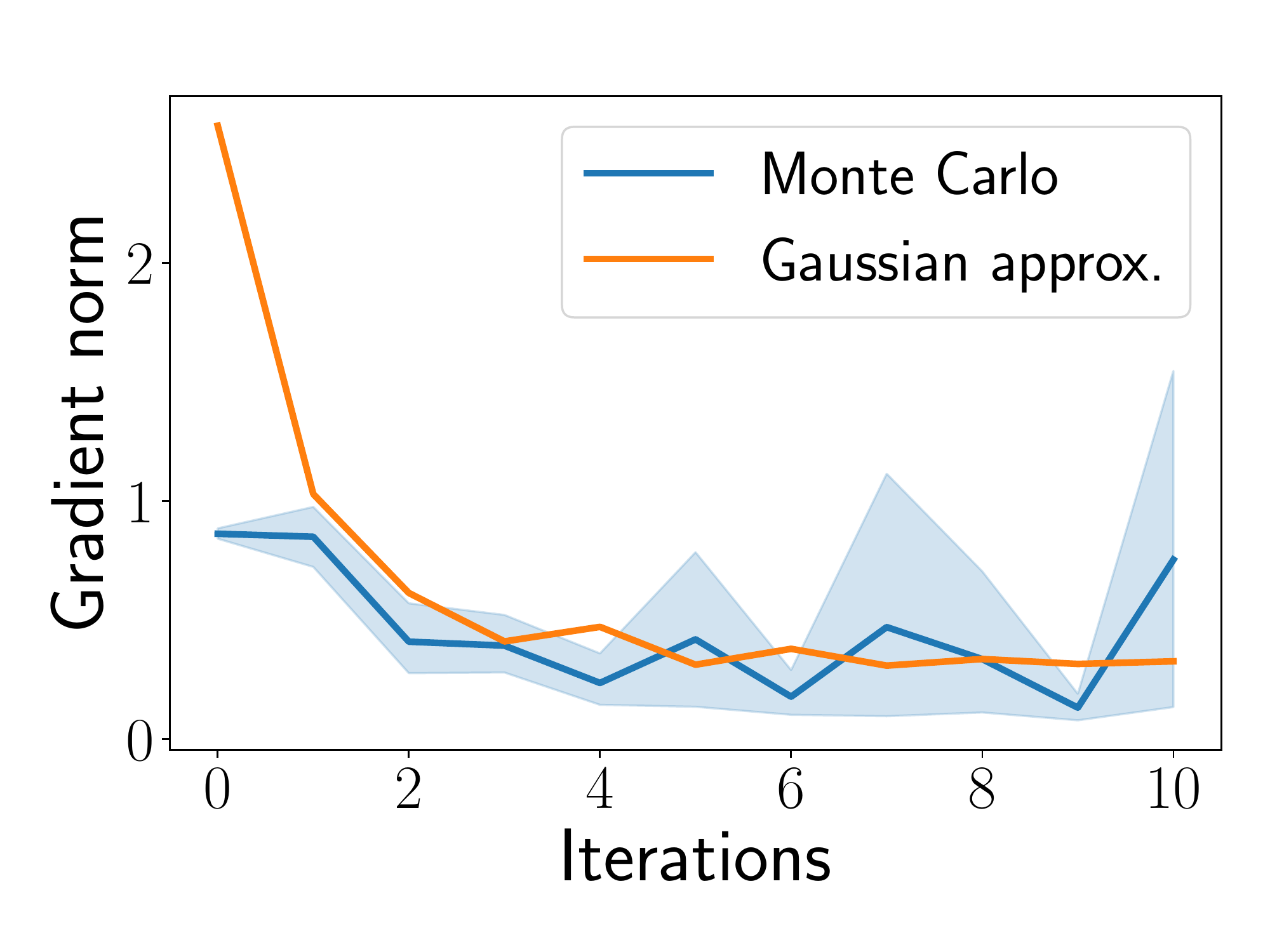}
		\caption{Risk-sensitive and gradient approximations. \label{fig:approx_risk}
		}
		\begin{subfigure}{0.335\linewidth}\hspace{-5pt}
		\includegraphics[width=\linewidth]{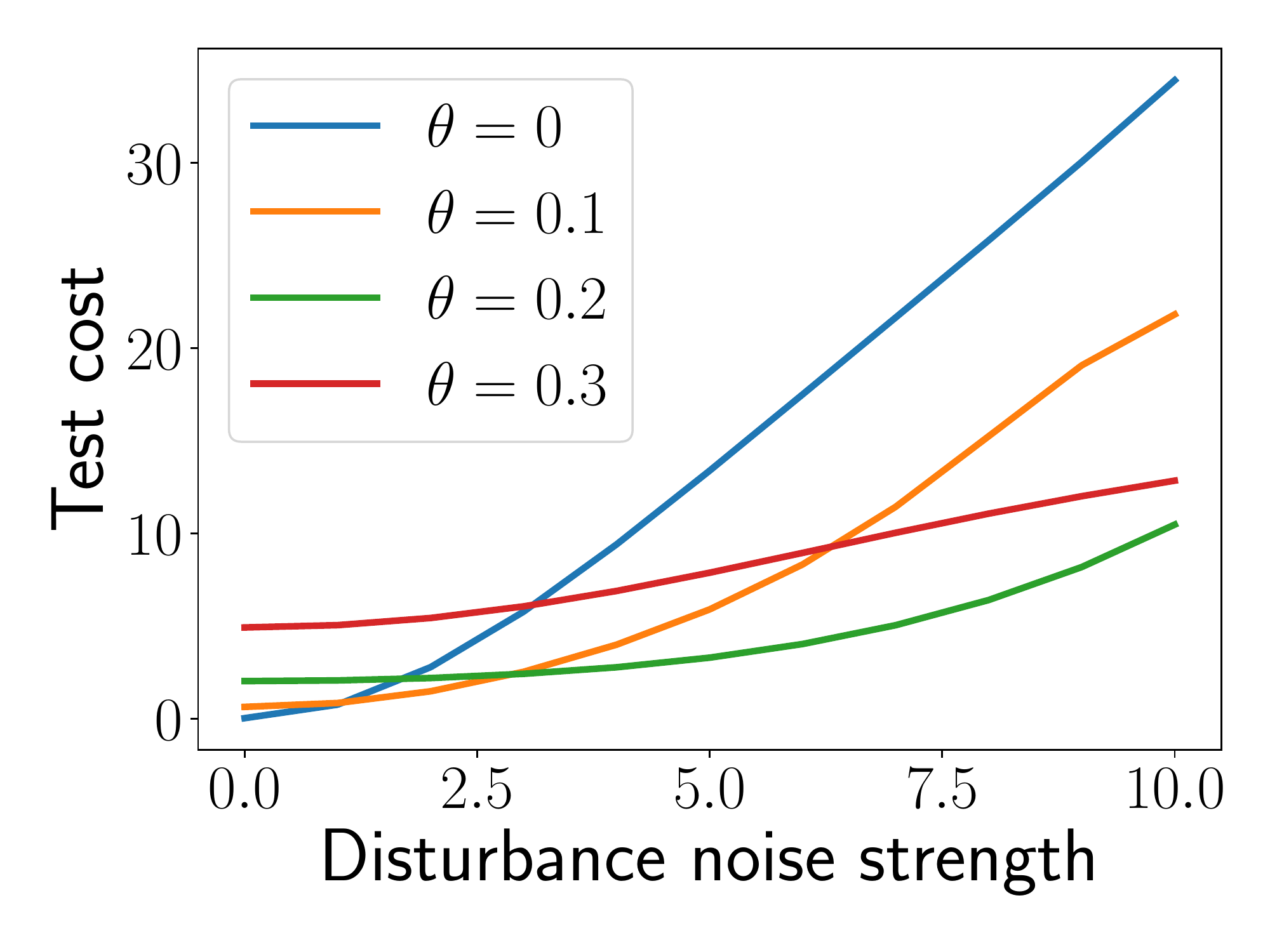}\vspace{-3pt}
		\caption{Pendulum.}
		\end{subfigure}~\hspace{-10pt}
		\begin{subfigure}{0.335\linewidth}
			\includegraphics[width=\linewidth]{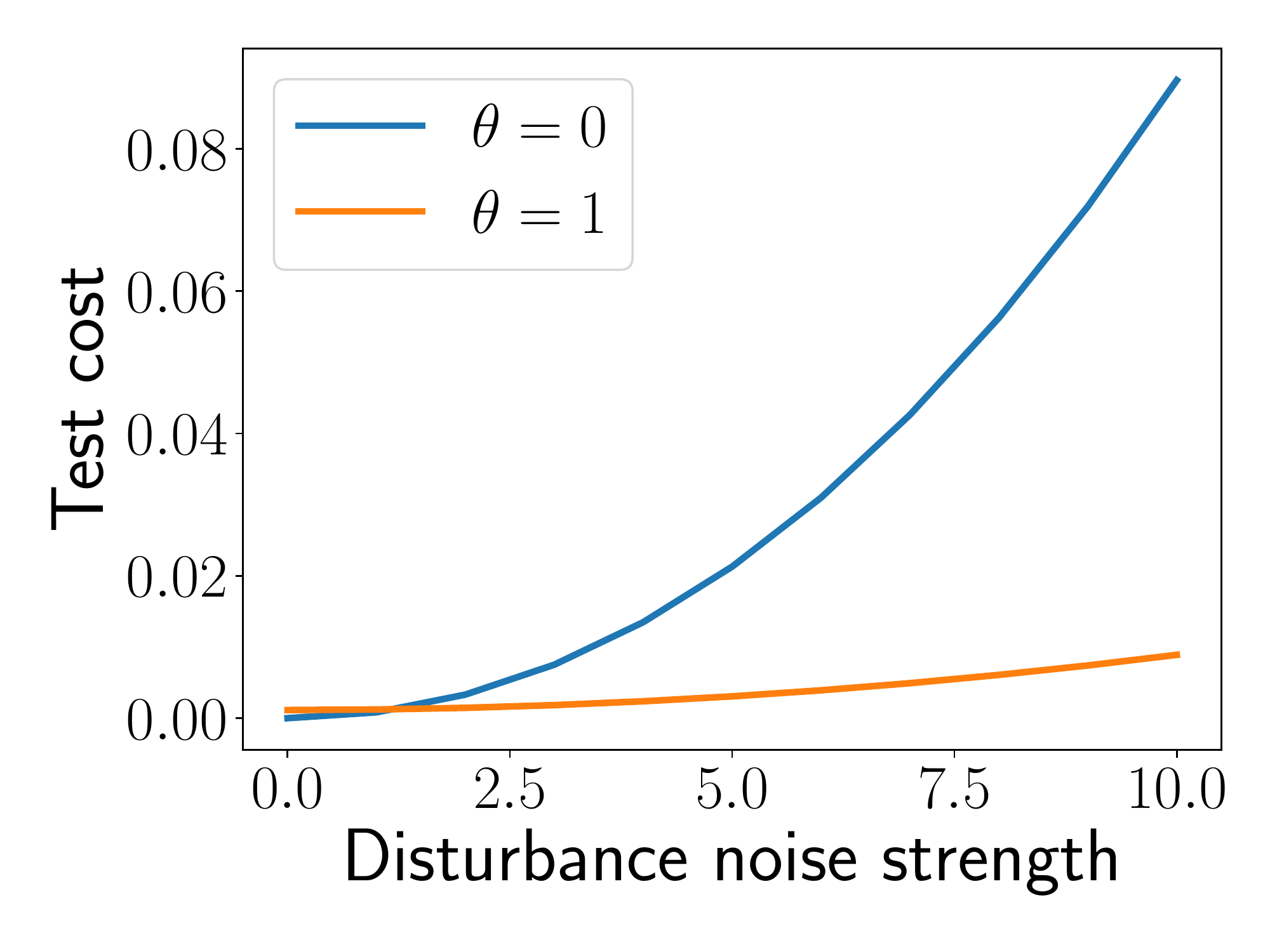}\vspace{-3pt}
			\caption{Two-link arm.}
		\end{subfigure}
		\vspace{1ex}
		\caption{Robustness of controllers against disturbance noise. \label{fig:robust}}
	\end{center}
\end{figure}

\paragraph{Convergence}
In Fig.~\ref{fig:conv} we compare the convergence on the pendulum problem of RegILEQG and ILEQG. For both algorithms, we use a constant step-size sequence tuned after a burn-in phase of 5 iterations on a grid of step-sizes $2^i$ for $i\in [-5,10]$. The surrogate risk-sensitive cost was used to tune the step-sizes. The best step-sizes found were $0.5$ for ILEQG and $16$ for RegILEQG. We plot the minimum values obtained until now, as the true function can be approximated. We observe that both ILEQG and RegILEQG minimize well the surrogate risk-sensitive cost. Yet, the regularized variant provides smoother convergence. We leave as future work the implementation of line-search procedures as done for Levenberg-Marquardt methods. 

\paragraph{Risk-sensitive cost approximation}
In Fig.~\ref{fig:approx_risk}, we compare $\hat f_\theta(\bar u^{(k)}), \|\nabla \hat  f_\theta(\bar u^{(k)})\|_2$ computed by the Gaussian approximation given in~\eqref{eq:approx_risk} and $f_\theta(\bar u^{(k)}), \|\nabla f_\theta(\bar u^{(k)})\|_2$ approximated by Monte-Carlo for $N=100$ samples and 10 runs. We plot these values along the iterations of the RegILEQG method for the pendulum (same experiment as in Fig.~\ref{fig:conv}). We observe that the approximation $\hat f_\theta(\bar u^{(k)})$ is close to the approximation by Monte-Carlo. 
The sequence of compositions defining the trajectory leads to highly non-smooth functions (i.e. large smoothness constants), which contributes to the high variance of gradients computed by Monte-Carlo. 

\paragraph{Robustness}
In Fig.~\ref{fig:robust}, we plot the test cost obtained by the expected or risk-sensitive optimizers on the movement perturbed by a Dirac of increasing strength. We use our RegILEQG algorithm with constant-step-size tuned after a burn-in phase. The risk-sensitive approach provides smaller costs against perturbed trajectories. On the two-link-arm problem, we did not observe significant changes when varying the risk-sensitivity parameter. We leave the analysis of the choice of the parameter for future work.

%% file: sections/06_ccl.tex
We dissected the ILEQG algorithm to understand its correct implementation, this revealed: (i) the objective it minimizes, that is not the risk-sensitive cost but an approximation of it, (ii) the necessary introduction from an optimization viewpoint of a regularization inside the step, (iii) a sufficient decrease condition that ensures proven stationary convergence to a near-stationary point. 

%% file: sections/07_acknowledgments.tex
This work was funded by NIH R01 (\#R01EB019335), NSF CPS (\#1544797), NSF NRI (\#1637748), NSF CCF (\#1740551), NSF DMS (\#1839371), DARPA Lagrange grant FA8650-18-2-7836, the program “Learning in Machines and Brains” of CIFAR, ONR, RCTA, Amazon, Google, Honda and faculty research awards.

%% file: appendix/01_notations.tex
\subsection{Miscellaneous}
We use semicolons to denote concatenation of vectors, namely for $n$ $d$-dimensional vectors $a_1, \ldots, a_n \in \reals^d$, we have $ (a_1;\ldots; a_n) \in \reals^{nd}$. The Kronecker product is denoted $\otimes$. For a sequence of matrices $X_1, \ldots X_\horizon \in \reals^{d \times p}$ we denote 
\[
\diag(X_1, \ldots, X_\horizon) = \left(\begin{matrix}
X_1 & 0 & \ldots & 0 \\
0 & \ddots & \ddots & \vdots \\
\vdots & \ddots & \ddots & 0 \\
0 & \ldots & 0& X_\horizon 
\end{matrix}\right) \in \reals^{d\horizon \times p \horizon}.
\]
the corresponding block diagonal matrix.
For a set $S\subset \reals^d$ and $x\in \reals^d$, denote $\dist(x, S)^2 = \min_{ y \in \reals^{d}}\|x-y\|_2^2 $. 
Given a density function $p: \reals^d \rightarrow \reals^+$, such that $\int_{\reals^d} p(w)dw = 1$ and a function $f: \reals^d \rightarrow \reals^p$ we denote
\[
\Expect_{w \sim p} f(w) = \int_{\reals^d} f(w)p(w)dw.
\]
For a random variable $w \in \reals^d$, we denote its covariance matrix by
\[
\operatorname{\mathbb{V}ar}(w) = \Expect((w-\Expect(w))(w-\Expect(w))^\top).
\]
For a matrix $M \in \reals^{d\times d}$, we denote $\|M\|_2 = \sup_{x \in \reals_*^d} x^\top Mx/ \|x\|_2^2$ the spectral norm induced by the Euclidean norm. 
We denote semi-definite positive matrices $S \in \reals^{d \times d}$ as $S\succeq 0$ and denote $\lambda_{\max}(S) = \|S\|_2$ the maximal eigenvalue of $S$. For a matrix $A\in \reals^{d \times n}$ we denote by $A^\dagger$ the pseudo-inverse of $A$.

\subsection{Tensors}
For a tensor $\mathcal{A} = (a_{i,j,k})_{\substack{i \in\{1, \ldots, d\},\; j \in\{1,\ldots,n\}, \; k \in \{1,\ldots,p\}}} \in \reals^{d \times n \times p}$, we denote $\mathcal{A}_{i,\cdot,\cdot} = (a_{i,j,k})_{\substack{j \in\{1,\ldots,n\}, \; k \in \{1,\ldots,p\}}} \in \reals^{n\times p}$ the matrix obtained by fixing the first index at $i$. Similarly we define $\mathcal{A}_{\cdot, j, \cdot} \in \reals^{d\times p}$ and $\mathcal{A}_{\cdot, \cdot, k} \in \reals^{d \times n}$. A tensor $\mathcal{A}$ can be represented as the list of matrices $\mathcal{A} = (\mathcal{A}_{\cdot, \cdot, 1},\ldots, \mathcal{A}_{\cdot, \cdot, k})$.
Given matrices $P \in \reals^{d \times d'}, Q \in \reals^{n \times n'}, R \in \reals^{p \times p'}$, we denote
\[
\mathcal{A}[P, Q, R] = \left(\sum_{k=1}^{p} R_{k,1}P^\top \mathcal{A}_{\cdot, \cdot, k} Q,  \ldots,  \sum_{k=1}^{p} R_{k,p'}P^\top \mathcal{A}_{\cdot, \cdot, k} Q \right) \in \reals^{d'\times n'\times p'}
\]
If $P, Q$ or $ R$ are identity matrices, we use the symbol "$\: \cdot\: $" in place of the identity matrix. For example, we denote $\mathcal{A}[P, Q, \idm_p] = \mathcal{A}[P,Q, \cdot] = \left(P^\top \mathcal{A}_{\cdot, \cdot, 1} Q,  \ldots,  P^\top\mathcal{A}_{\cdot, \cdot, p}Q \right)$.
If $P, Q$ or $R$ are vectors we consider the flatten object. In particular, for $x\in \reals^d, y\in \reals^n$, we denote
\[
\mathcal{A}[x, y, \cdot] =  \left(\begin{matrix}
x^\top\mathcal{A}_{\cdot, \cdot, 1} y\\ \vdots \\ x^\top \mathcal{A}_{\cdot, \cdot, p}y
\end{matrix}
\right) \in \reals^p
\]
rather than having $\mathcal{A}[x, y, \cdot] \in \reals^{1 \times 1\times p}$.
Similarly, for $z\in \reals^p$, we have
\[
\mathcal{A}[\cdot, \cdot, z] = \sum_{k=1}^p z_k\mathcal{A}_{\cdot, \cdot, k} \in \reals^{d\times n}.
\]

For a tensor $\mathcal{A}$, we denote
\begin{equation}\label{eq:tensor_norm}
\|\mathcal{A}\|_2 = \sup_{x \in \reals^d_*, y \in \reals^n_*, z\in \reals^p_*} \frac{\mathcal{A}[x, y, z]}{\|x\|_2\|y\|_2\|z\|_2}
\end{equation}
the norm induced by the Euclidean norm for the tensor $\mathcal{A}$.

\subsection{Gradients}
For a multivariate function $f :\reals^d \mapsto \reals^n$, composed of $f^{(j)}$ real functions with $j\in \{1, \ldots, n\}$, we denote $\nabla f( x) = (\nabla f^{(1)}( x), \ldots, \nabla f^{(n)}( x))\in \reals^{d \times n}$, that is  the transpose of its Jacobian on $ x$, $\nabla f( x) = (\frac{\partial f^{(j)} }{\partial x_i}( x))_{\substack{1\leq i\leq d, 1\leq j\leq n}} \in \reals^{d \times n}$. 
We represent its 2nd order information by a tensor $\nabla^2 f( x) = (\nabla^2 f^{(1)}( x), \ldots, \nabla^2 f^{(n)}( x))\in \reals^{d \times d \times n}$

For a real function, $f: \reals^d \times \reals^p \mapsto \reals$, whose value is denoted $f(x,y)$, we decompose its gradient $\nabla f( x,  y) \in \reals^{d +p}$ on $( x,  y) \in \reals^d \times \reals^p$ as 
\[
\nabla f( x,  y) =\left(
\begin{matrix}
\nabla_x f( x,  y) \\
\nabla_y f( x,  y)
\end{matrix}\right) \qquad \mbox{with} \qquad \nabla_x f(x,y) \in \reals^d, \quad \nabla_y f(x,y) \in \reals^p.
\]
For a multivariate function $f: \reals^d \times \reals^p  \mapsto \reals^{n}$ and $( x,  y)$, we denote $\nabla_x f( x, y ) = (\nabla_x f^{(1)}( x,  y), \ldots, \nabla_x f^{(n)}( x,  y))\in \reals^{d \times n}$ and  we define similarly $\nabla_y f( x,  y)\in \reals^{p \times n}$. 

We drop the dependency to the time when it is clear from context, e.g., for a dynamic $\phi_t:\reals^{d+p}\rightarrow \reals^d$ we denote by $\nabla_u \phi_t(x_t, u_t) = \nabla_{u_t} \phi_t(x_t, u_t)$.
Those definitions extend for noisy dynamics $\dyn_t$, where we add the noise variable $w\in \reals^q$.

All Lipschitz continuity constants are defined w.r.t. the norm induced by the Euclidean norm. In particular, for a multivariate twice differentiable function $f$, we say that it is smooth if its second-order tensor has a bounded norm for the Euclidean induced norm of a tensor defined in~\eqref{eq:tensor_norm}.

%% file: appendix/02_linear_quad_risk_ctrl.tex
\subsection{Min-max formulation}
\linquadrisk*
\begin{proof}[Proof of (i)]
	Since $w_t$ are i.i.d, the states $x_t$ given by the linear dynamics form a Markov sequence of random variables, i.e., denoting $\Prob$ the probability defined by the dynamics, for any $t\in \{0,\ldots,\horizon-1\}$, $\Prob(x_{t+1}|x_t,\ldots, x_0) = \Prob(x_{t+1}|x_t) \sim \mathcal{N}(A_tx_t + B_tu_t, \Sigma_t ) $ where $\Sigma_t = \sigma^2C_t C_t^\top$ and $x_0 = \hat x_0$. Since $\Sigma_t$ is potentially not full-ranked, the probability distribution of $\bar x$ requires to define an appropriate measure. Denote $\Pi_{\operatorname{Null}(\Sigma_t)}$ the orthonormal projection on the null space of $\Sigma_t$ and denote by $\mu$ any measure such that
	\begin{align*}
		d\mu(\bar x) = \begin{cases}
		0 & \mbox{if}\ \exists t \in\{0,\ldots\horizon-1\}:  \Pi_{{\operatorname{Null}(\Sigma_t)}}(x_{t+1} - A_tx_t - B_tu_t) \neq 0, \\
		d\lambda(\bar x) & \mbox{otherwise},
		\end{cases}
	\end{align*}
	where $d\lambda(\bar x)$ is the Lebesgue measure on $\reals^{\horizon d}$. Therefore, we have
	\begin{align*}
	\Expect_{\bar x\sim p(\cdot; \bar u)} \left[ \exp( \theta h(\bar x))\right] & \propto \int \exp\bigg( - \sum_{t=0}^{\horizon-1} \frac{1}{2}(x_{t+1} - A_t x_t - B_t u_t)^\top \Sigma^{\dagger}_t(x_{t+1} - A_t x_t - B_t u_t)  \\
	& \phantom{ = \int \exp\bigg( } + \theta \sum_{t=1}^{\horizon} \frac{1}{2} x_t^\top H_t x_t + \tilde h_t^\top x_t\bigg) d\mu(\bar x) \\
	& =  \int \exp(-q(\bar x, \bar u) )d\mu(\bar x),
	\end{align*}
	where  $q(\bar x, \bar u)$ is a quadratic in $\bar x, \bar u$ and we ignored the normalization constants in the first line as we are interested in computing the minimum. Fix $\bar u$ and denote simply $\tilde q(\bar x) = q(\bar x, \bar u)$. The integral will then be finite if and only if $\tilde q(\bar x)$ is bounded below in $\bar x \in\mathcal{X} = \{\bar x: \forall t\in\{0, \ldots,\horizon-1\}\  \Pi_{\operatorname{Null}(\Sigma_t)}(x_{t+1} - A_tx_t - B_tu_t) =0\}$. In that case, denote $\bar x^* \in \argmin_{\bar x\in \mathcal{X}} \tilde q(\bar x)$, using the Taylor expansion of $\tilde q$, we get for $\bar x \in \mathcal{X}$, $\tilde q(\bar x) =  \tilde q(\bar x_*) + \frac{1}{2}(\bar x - \bar x_*)^\top Q(\bar x - \bar x_*)$ where $Q = \nabla^2 \tilde q(\bar x)$ is independent of $\bar x, \bar u$ and we use that $\nabla\tilde q(\bar x^*)^\top(\bar x -\bar x^*) = 0$ for $\bar x \in \mathcal{X}$ by definition of $\bar x^*$. The expectation is then proportional to, the variance term defined by $Q$ being independent of $\bar u$,
	\begin{align*}
	\Expect_{\bar x\sim p(\cdot;\bar u)} \left[ \exp(\theta h(\bar x))\right] \propto  \exp\left(-\min_{\bar x} q(\bar x, \bar u)\right).
	\end{align*}
	By parameterizing the states as $x_{t+1} = A_tx_t + B_tu_t + C_tw_t$ for $\bar x \in \mathcal{X}$, using that $C_t$ has the same image as $\Sigma_t$, the minimization can be rewritten
	\begin{align*}
	\min_{\bar x \in \mathcal{X}} q(\bar x, \bar u)  =  \qquad \min_{\bar w \in \reals^{\horizon d}, \bar x \in \reals^{\horizon d}} \quad & -\theta \sum_{t=1}^{\horizon} \left(\frac{1}{2} x_t^\top H_t x_t +\tilde h_t^\top x_t\right) + \sum_{t=0}^{\horizon-1} \frac{1}{2 \sigma ^2}\|w_t\|_2^2\\
	\mbox{subject to} \quad & x_{t+1} = A_t x_t + B_t u_t + C_t w_t \\
	& x_0 = \hat x_0.
	\end{align*}
	The risk sensitive control problem~\eqref{eq:risk_ctrl} is then equivalent to, i.e., shares the same set of minimizers as,
	\begin{align*}
	\min_{\bar u\in \reals^{\horizon p}} \sup_{\bar w \in \reals^{\horizon q}, \bar x\in \reals^{\horizon d}} \quad & \ \sum_{t=1}^{\horizon}\frac{1}{2}x_t^\top H_t x_t + \tilde h_t^\top x_t + \sum_{t=0}^{\horizon-1} \frac{1}{2}u_t^\top G_t u_t + \tilde g_t^\top u_t - \sum_{t=0}^{\horizon-1} \frac{1}{2\theta \sigma ^2}\|w_t\|_2^2 \\
	\mbox{\textup{subject to}} \quad & x_{t+1} = A_tx_t +B_tu_t +C_tw_t \nonumber\\
	& x_0 = \hat x_0, \nonumber
	\end{align*}
	which, if the sup is infinite, means that the problem is not defined.
\end{proof}
\begin{proof}[Proof of (ii)]
	The linear dynamics read $x_{t+1} - A_tx_t = B_t u_t + C_t w_t$ for $t=0, \ldots, \horizon-1$. Denoting
	\[
	L = 
	\left(\begin{matrix}
	I & 0 & \ldots & 0\\
	-A_1 & I & \ddots & \vdots\\
	\vdots & \ddots & \ddots & 0\\
	0 & \ldots & -A_{\horizon-1} & I 
	\end{matrix}\right)  
	\qquad
	\mbox{with}
	\qquad
	L^{-1} =
	\left(\begin{matrix}
	I & 0 & \ldots& 0\\
	A_1 & I & 0& 0 \\
	\vdots & \vdots& \ddots& \vdots\\
	A_{\horizon-1}\ldots A_1 & A_{\horizon-1}\ldots A_2  & \ldots & I
	\end{matrix}\right),
	\] 
	we get
	\[
	L \bar x = \bar B \bar u + \bar C\bar w + \breve x_0 \quad \mbox{and so} \quad 
	\bar x = L^{-1} (\bar B \bar u + \bar C\bar w + \breve x_0),
	\]
	where $\breve x_0 = (A_0 \hat x_0; 0;\ldots; 0) \in \reals^{\horizon d}$, $\bar x = (x_1; \ldots; x_{\horizon})$, $\bar B = \diag(B_0,\ldots, B_{\horizon-1})$, $\bar C = \diag(C_0,\ldots, C_{\horizon-1})$.
	Problem~\eqref{eq:min_max_ctrl} reads then
	\begin{align}
	\min_{\bar u \in \reals^{\horizon p}} \sup_{\bar w \in \reals^{\horizon q}} & \frac{1}{2} (\bar B \bar u + \bar C \bar w + \breve x_0)^\top L^{-\top} \bar H L^{-1} (\bar B \bar u + \bar C \bar w + \breve x_0) + \bar h^\top L^{-1} (\bar B \bar u + \bar C \bar w + \breve x_0) \label{eq:lin_quad_simp} \\
	& + \frac{1}{2}\bar u^\top \bar G \bar u  + \bar g^\top \bar u- \frac{1}{2\theta \sigma ^2}\|\bar w\|_2^2, \nonumber
	\end{align}
	where $	\bar H = \diag(H_1,\ldots, H_{\horizon}), \bar G = \diag(G_0,\ldots, G_{\horizon-1}), \bar h = (h_1;\ldots; h_\horizon)$ and $\bar g = (g_0;\ldots; g_{\horizon-1})$.
	It is always a strongly convex problem in $\bar u$ by assumption on the $G_t$. If 
	\[
(\theta\sigma^2)^{-1} <  \lambda_{\max}(\bar C^\top L^{-\top} \bar H L^{-1} \bar C),
	\]
	i.e., $(\theta\sigma^2)^{-1} \id_{\horizon q} \not\succeq \bar C^\top L^{-\top} \bar H L^{-1} \bar C$, then there exists $\bar w^*$ such that $\bar{w^*}^\top (\bar C^\top L^{-\top} \bar H L^{-1} \bar C - (\theta\sigma^2)^{-1}\id_{\horizon q})\bar w^* > 0$, by taking $\alpha \bar w^*$ in place of $\bar w^*$ with $\alpha\rightarrow +\infty$, the maximization problem in~\eqref{eq:lin_quad_simp} is always infinite, independently of $\bar u$.
 The claim follows by identifying $H = \nabla^2 h(\bar x) = \bar H$, $\tilde C = L^{-1} \bar C$ and $\tilde x_0  = L^{-1}\breve x_0$.
\end{proof}
\begin{proof}[Proof of (iii)]
		If 
	\begin{equation}\label{eq:cond_risk_linear}
(\theta\sigma^2)^{-1}  >  \lambda_{\max}(\bar C^\top L^{-\top} \bar H L^{-1} \bar C), 
	\end{equation}
	i.e., $(\theta\sigma^2)^{-1} \id_{\horizon q} \succ  \bar C^\top L^{-\top} \bar H L^{-1} \bar C$, the maximization problem in~\eqref{eq:lin_quad_simp} is a strongly concave problem in $\bar w$ such that the sup on $\bar w$ is attained. For the dynamic programming resolution, define cost-to-go functions starting from $y$ at time $t$ as
	\begin{align*}
	\costfunc_t(y) = \min_{u_t,\ldots, u_{\horizon-1}} \sup_{\substack{w_t,\ldots, w_{\horizon -1}\\ x_t, \ldots, x_\horizon}} &
	\sum_{s=t}^{\horizon}\frac{1}{2}x_s^\top H_s x_s + \tilde h_s^\top x_s + \sum_{s=t}^{\horizon-1} \frac{1}{2}u_s^\top G_s u_s + \tilde g_s^\top u_s - \sum_{s=t}^{\horizon-1} \frac{1}{2\theta\sigma^2}\|w_s\|_2^2  \\
	\mbox{\textup{subject to}} \quad & x_{s+1} = A_sx_s +B_su_s + C_s w_s \quad \mbox{for $s=t,\ldots, \horizon-1$}\nonumber\\
	& x_t = y, \nonumber
	\end{align*}
	with the convention $H_0 = 0$, $\tilde h_0=0$. Cost-to-go functions satisfy the Bellman equation
	\begin{align}
	\costfunc_t(y) =  \frac{1}{2}y^\top H_t y + \tilde h_t^\top y  + \min_{u_t \in \reals^p} \sup_{w_t\in \reals^q} \bigg\{ & \frac{1}{2} u_t^\top G_t u_t + \tilde g_t^\top u_t- \frac{1}{2\theta\sigma^2} \|w_t\|_2^2  + \costfunc_{t+1}(A_t y + B_t u_t + C_t w_t)\bigg\}, \label{eq:Bellman}
	\end{align}
	with optimal control 
	\begin{align*}
	u^*_t(y) =  \argmin_{u_t \in \reals^p} \bigg\{&  \frac{1}{2} u_t^\top G_t u_t + \tilde g_t^\top u_t   + \sup_{w_t\in \reals^q} \Big\{- \frac{1}{2\theta\sigma^2}\|w_t\|_2^2  + \costfunc_{t+1}(A_t y + B_t u_t + C_tw_t)\Big\}\bigg\},
	\end{align*}
	and optimal noise, if the sup is finite,
	\[
	w_t^*(u_t, y) = \argmax_{ w_t\in \reals^{ d}}\Big\{- \frac{1}{2\theta\sigma^2}\|w_t\|_2^2  + \costfunc_{t+1}(A_t y + B_t u_t + C_tw_t)\Big\}.
	\]
	The final cost initializing the recursion is defined as $\costfunc_\horizon(y) = \frac{1}{2}y^\top H_\horizon y + \tilde h_\horizon^\top y$. For quadratic costs and linear dynamics, the cost-to-go functions are quadratic and can be computed analytically through the recursive equation~\eqref{eq:Bellman}. If the quadratic defining the supremum problem is not negative semi-definite the problem is infeasible.
	
	If condition~\eqref{eq:cond_risk_linear} holds, the overall maximization is feasible, all suprema are reached. The solution of~\eqref{eq:min_max_ctrl} is given by computing $\costfunc_0(\hat x_0)$, which amounts to solve iteratively the Bellman equations starting from $x_0 = \hat x_0 $, i.e.,  getting the optimal control at the given state and moving along the dynamics to compute the next cost-to-go:
	\begin{align*}\label{eq:roll_out}
	u^*_t = u^*_t(x_t), \quad w_t^* = w_t^*(u_t^*, x_t), \quad
	x_{t+1} = A_t x_t + B_t u_t^* + C_t w_t^*.
	\end{align*}
\end{proof}

\subsection{Dynamic programming resolution}
Detailed computations of the dynamic programming approach are given in the following proposition that supports Algo.~\ref{algo:LEQG}.
Though finer sufficient conditions to get a solution can be derived in the case $(\theta\sigma^2)^{-1} = \lambda_{\max}(C_t^\top P_{t+1}C_t)$, simply reducing the risk sensitivity parameter is enough to get the condition in line~\ref{state:dyn_prog_cond}. For simplicity, in Algo.~\ref{algo:LEQG}, if condition~\eqref{eq:dyn_prog_cond_app} is not satisfied, we consider the problem to be infeasible.

\begin{proposition}\label{prop:dyn_prog_lin_quad_risk_ctrl}
	Consider Algo.~\ref{algo:LEQG} applied for the linear quadratic risk sensitive control problem~\eqref{eq:min_max_ctrl} with $H_t\succeq 0$ and $G_t\succ 0$. If condition
	\begin{equation}\label{eq:dyn_prog_cond_app}
	(\theta\sigma^2)^{-1} > \lambda_{\max}(C_t^\top P_{t+1}C_t)
	\end{equation}
	in line~\ref{state:dyn_prog_cond} is satisfied for all $t=\horizon-1,\ldots, 0$, then the cost-to-go functions are quadratics of the form 
	\begin{equation}\label{eq:dyn_prog_rec_app}
	\costfunc_t(y) = \frac{1}{2} y^\top P_t y + p_t^\top y + c\quad \mbox{with} \quad P_t \succeq 0,
	\end{equation}
	where $c$ is a constant and $P_t, p_t$ are defined recursively in line~\ref{state:dyn_prog_rec}.
	
	If for any $t = \horizon-1,\ldots, 0$,
	\[
	(\theta\sigma^2)^{-1} < \lambda_{\max}(C_t^\top P_{t+1}C_t),
	\]
	the linear quadratic risk sensitive control problem~\eqref{eq:min_max_ctrl} is infeasible.
\end{proposition}
\begin{proof}
	The cost-to-go function at time $\horizon$ reads $\costfunc_\horizon(y) = \frac{1}{2}y^\top H_\horizon y + \tilde h_\horizon^\top y$. It has then the form~\eqref{eq:dyn_prog_rec_app} with $p_\horizon = \tilde h_\horizon$  and $P_\horizon= H_\horizon \succeq 0$.
	Assume now that at time $t+1$, the cost-to-go function has the form of~\eqref{eq:dyn_prog_rec_app}, i.e., $\costfunc_{t+1}(y) = \frac{1}{2}y^\top P_{t+1} y + p_{t+1}^\top y$ with $P_{t+1}\succeq 0$. Then, the Bellman equation reads, ignoring the constant terms,
	\begin{align*}
		\costfunc_t(y) & = \frac{1}{2}y^\top H_t y + \tilde h_t^\top y+ \min_{u_t \in \reals^p}\sup_{w_t\in \reals^q}\bigg\{ \frac{1}{2} u_t^\top G_t u_t +\tilde g_t^\top u_t -  \frac{1}{2\theta\sigma^2} \|w_t\|_2^2 \\
		& \phantom{= \frac{1}{2}y^\top H_t y + \tilde h_t^\top y + \min_{u_t \in \reals^p}\sup_{w_t\in \reals^d}\bigg\{} + p_{t+1}^\top (A_t y + B_t u_t + C_t w_t)\\
		& \phantom{= \frac{1}{2}y^\top H_t y + \tilde h_t^\top y + \min_{u_t \in \reals^p}\sup_{w_t\in \reals^d}\bigg\{}
		+  \frac{1}{2}(A_t y + B_t u_t + C_t w_t)^\top P_{t+1}(A_t y+ B_t u_t + C_t w_t)\bigg\}  \\
		& = \frac{1}{2}y^\top H_t y + \tilde h_t^\top y + \min_{u_t \in \reals^p} \bigg\{\frac{1}{2} u_t^\top G_t u_t  + \tilde g_t^\top u_t \\
		&  \phantom{= \frac{1}{2}y^\top H_t y + \tilde h_t^\top y + \min_{u_t \in \reals^p}\bigg\{}
		+ \frac{1}{2}(A_t y + B_t u_t)^\top P_{t+1}(A_t y + B_t u_t) + p_{t+1}^\top(A_t y + B_t u_t)  \\ 
		& \phantom{= \frac{1}{2}y^\top H_t y + \tilde h_t^\top y + \min_{u_t \in \reals^p}\bigg\{}
		+ \sup_{w_t\in \reals^q}\bigg[\frac{1}{2}w_t^\top C_t^\top [P_{t+1}(A_t y + B_t u_t) + p_{t+1}] \\
		& \phantom{= \frac{1}{2}y^\top H_t y + \tilde h_t^\top y + \min_{u_t \in \reals^p}\bigg\{ + \sup_{w_t\in \reals^d}\bigg[}
		- \frac{1}{2}w_t^\top((\theta\sigma^2)^{-1} \id_q - C_t^\top P_{t+1}C_t)w_t\bigg]\bigg\} .
	\end{align*}
	If $(\theta\sigma^2)^{-1} < \lambda_{\max}(C_t^\top P_{t+1}C_t)$, the supremum in $w_t$ is infinite. If $(\theta\sigma^2)^{-1} > \lambda_{\max}(C_t^\top P_{t+1}C_t)$, the supremum is finite and reads
	\begin{equation}
	w_t^* = ((\theta\sigma^2)^{-1} \id_q - C_t^\top P_{t+1}C_t)^{-1}C_t^\top [P_{t+1}(A_t y + B_t u_t) + p_{t+1}]. \label{eq:opt_noise}
	\end{equation}
	So we get, ignoring the constant terms,
	\begin{align}\label{eq:Bellman_quad_sol}
		\costfunc_t(y) = \frac{1}{2}y^\top H_t y + \tilde h_t^\top y + \min_{u_t \in \reals^p}\Big\{ & \frac{1}{2}u_t^\top G_tu_t +\tilde g_t^\top u_t  \nonumber\\
		& + \frac{1}{2}(A_t y + B_t u_t)^\top \tilde P_{t+1}(A_t y + B_t u_t) + \tilde p_{t+1}^\top(A_t y + B_t u_t)\Big\}, 
	\end{align}
	where 
	\begin{align*}
		\tilde P_{t+1} & = P_{t+1} + P_{t+1}C_t((\theta\sigma^2)^{-1} \id_q - C_t^\top  P_{t+1}C_t)^{-1}C_t^\top P_{t+1} \succeq 0 \\
		\tilde p_{t+1} & = p_{t+1} + P_{t+1}C_t((\theta\sigma^2)^{-1} \id_q - C_t^\top  P_{t+1}C_t)^{-1}C_t^\top p_{t+1}.
	\end{align*}
	We then get, ignoring the constant terms,
	\begin{align*}
		\costfunc_t(y) = & \frac{1}{2}y^\top (H_t + A_t^\top \tilde P_{t+1}A_t)y  + (\tilde h_t + A_t^\top \rho_t)^\top y - \frac{1}{2}y^\top A_t^\top\tilde P_{t+1}B_t (G_t + B_t^\top\tilde P_{t+1}B_t)^{-1} B_t^\top \tilde P_{t+1} A_t y.
	\end{align*}
	where $\rho_t = \tilde p_{t+1} - \tilde P_{t+1}B_t (G_t + B_t^\top\tilde P_{t+1}B_t)^{-1} [B_t^\top \tilde p_{t+1} +\tilde g_t]$.
	The cost function is then a quadratic defined by
	\begin{align*}
		P_t & = H_t + A_t^\top \tilde P_{t+1}A_t - A_t^\top\tilde P_{t+1}B_t (G_t + B_t^\top\tilde P_{t+1}B_t)^{-1} B_t^\top \tilde P_{t+1} A_t.
	\end{align*}
	Denoting $\tilde P_{t+1}^{\nicefrac{1}{2}}$ a square root matrix of $\tilde P_{t+1}$ such that $\tilde P_{t+1}^{\nicefrac{1}{2}} \succeq 0$ and $\tilde P_{t+1}^{\nicefrac{1}{2}}\tilde P_{t+1}^{\nicefrac{1}{2}} = \tilde P_{t+1}$, we get
	\begin{align*}
		P_t & = H_t + A_t^\top \tilde P_{t+1}^{\nicefrac{1}{2}}\big( \id_d -  \tilde P_{t+1}^{\nicefrac{1}{2}} B_t (G_t +  B_t^\top \tilde P_{t+1} B_t)^{-1} B_t^\top \tilde P_{t+1}^{\nicefrac{1}{2}}\big)  \tilde P_{t+1}^{\nicefrac{1}{2}} A_t \\
		&  = H_t +  A_t^\top \tilde P_{t+1}^{\nicefrac{1}{2}}\big(\id_d + \tilde P_{t+1}^{\nicefrac{1}{2}} B_t G_t^{-1} B_t^\top \tilde P_{t+1}^{\nicefrac{1}{2}}
		\big)^{-1}  \tilde P_{t+1}^{\nicefrac{1}{2}} A_t \succeq 0,
	\end{align*}
	where we use Sherman-Morrison-Woodbury formula for the last equality. This proves that $\costfunc_t(y)$ satisfies~\eqref{eq:dyn_prog_rec_app} at time $t$ with $P_t$ defined above and 
	\[
	p_t = \tilde h_t + A_t^\top \Big(\tilde p_{t+1} - \tilde P_{t+1}B_t (G_t + B_t^\top\tilde P_{t+1}B_t)^{-1} [B_t^\top \tilde p_{t+1} +\tilde g_t]\Big).
	\]
	The optimal control is given from~\eqref{eq:Bellman_quad_sol} as
	\[
	u^*_t(y) = -(G_t + B_t^\top\tilde P_{t+1}B_t)^{-1} [B_t^\top \tilde P_{t+1} A_t y + \tilde g_t + B_t^\top \tilde p_{t+1}]
	\]
	and the optimal noise is given by~\eqref{eq:opt_noise}, i.e., 
	\[
	w_t^*(y, u_t) = ((\theta\sigma^2)^{-1} \id_q - C_t^\top P_{t+1}C_t)^{-1}C_t^\top [P_{t+1}(A_t y + B_t u_t) + p_{t+1}]. 
	\]
\end{proof}

\begin{remark}
	Consider the case $\tilde h_t = 0$, $\tilde g_t = 0$ such that $\tilde p_{t+1} = 0$ and $p_{t+1} =0$. Then 
	Algorithm~\ref{algo:LEQG} is a modified version of the classical Linear Quadratic Regulator (LQR) algorithm where the value function at time $t+1$ is  $\tilde \costfunc_{t+1}(y) = y^\top \tilde P_{t+1} y/2$ instead of $ \costfunc_{t+1}(y) = y^\top P_{t+1} y/2$ for the LQR derivations. 
	
	In particular, denoting $P_{t+1}^{\nicefrac{1}{2}}$ a square root matrix of $P_{t+1}$ and using Sherman-Morrison-Woodbury formula, we have that 
	\begin{align*}
		\tilde P_{t+1} & = P_{t+1}^{\nicefrac{1}{2}}\big(\id_d - P_{t+1}^{\nicefrac{1}{2}}C_t(C_t^\top P_{t+1}C_t - (\theta\sigma^2)^{-1} \id_d)^{-1}C_t^\top P_{t+1}^{\nicefrac{1}{2}}\big) P_{t+1}^{\nicefrac{1}{2}} \\
		& = P_{t+1}^{\nicefrac{1}{2}}(\id_d - \theta\sigma^2 P_{t+1}^{\nicefrac{1}{2}} C_tC_t^\top P_{t+1}^{\nicefrac{1}{2}})^{-1}P_{t+1}^{\nicefrac{1}{2}},
	\end{align*}
	such that for $\theta = 0$ we get $\tilde P_{t+1} =P_{t+1}$, so we retrieve the minimization of a Linear Quadratic Gaussian control problem by dynamic programming.
\end{remark}

%% file: appendix/03_algos_app.tex
\subsection{Model minimization}
We present the implementation of RegILEQG for general noisy dynamics of the form
\begin{equation}\label{eq:noisy_dyn_gen}
x_{t+1} = \dyn_t(x_t,u_t, w_t).
\end{equation}
We define the trajectory as a function $\traj : \reals^{\horizon p \times \horizon q} \rightarrow \reals^{\horizon d}$ of the control and noise variables decomposed as $\traj(\bar u, \bar w) = (\traj_1(\bar u, \bar w); \ldots; \traj_\horizon(\bar u, \bar w))$ where
\begin{equation}\label{eq:traj_def_gen}
\traj_1(\bar u, \bar w) = \dyn_0(\hat x_0, u_0, w_0),  \quad \traj_{t+1}(\bar x, \bar w) = \dyn_t(\traj_t(\bar u, \bar w), u_t, w_t).
\end{equation}
The risk sensitive objective~\eqref{eq:risk_ctrl} can be written
\begin{align}
\min_{\bar u \in \reals^{\horizon p}} f_\theta(\bar u) &= \eta_{\theta}(\bar u) 
+ g(\bar u) \quad \mbox{where} \quad \eta_{\theta}(\bar u) = \frac{1}{\theta}\log \Expect_{\bar w} \Big[ \exp \theta h\big(\traj(\bar u, \bar w)\big)\Big] . \label{eq:non_lin_risk_ctrl_gen}
\end{align}
The model we consider for the trajectory reads
\begin{align}
	\traj(\bar u + \bar v, \bar w) & \approx \traj(\bar u, 0) + \nabla \traj(\bar u, 0)^\top(\bar v, \bar w) = \traj(\bar u, 0) + \nabla_{\bar u} \traj(\bar u, 0)^\top \bar v + \nabla_{\bar w} \traj (\bar u, 0)^\top \bar w, \label{eq:approx_traj_gen}
\end{align}
where $\traj(\bar u, 0)$ is the noiseless trajectory, $\nabla_{\bar u} \traj$ and $\nabla_{\bar w} \traj$ denote  the gradient w.r.t. the command and the noise, respectively, see Appendix~\ref{sec:notations} for gradient notations.

We approximate the objective as $f_\theta(\bar u + \bar v) \approx m_{f_\theta}(\bar u + \bar v ; \bar u) $, where 
\begin{align}
	m_{f_\theta}(\bar u + \bar v ; \bar u) \triangleq & \frac{1}{\theta}\log \Expect_{\bar w} \Big[ \exp \theta q_h\big(\bar x + \nabla_{\bar u} \traj(\bar u, 0)^\top \bar v + \nabla_{\bar w}\traj(\bar u, 0)^\top \bar w  ; \bar x  \big) \Big]  + q_g(\bar u + \bar v; \bar u), \label{eq:model_gen}
\end{align} 
where $q_h(\bar x + \bar y; \bar x) \triangleq h(\bar x) + \nabla h(\bar x)^\top \bar y + \bar y^\top\nabla^2 h(\bar x) \bar y/2$, $q_g(\bar u + \bar v; \bar u)$ is defined similarly and $\bar x = \traj(\bar u, 0)$ is the noiseless trajectory.

This model is then minimized with an additional proximal term. Formally, the algorithm starts at a point $\bar u^{(0)}$ and  defines the next iterate as
\begin{equation}\label{eq:min_model_step_gen}
\bar u^{(k+1)} = \bar u^{(k)}  + \argmin_{\bar v \in \reals^{\horizon p}}\left\{ m_{f_\theta}(\bar u^{(k)} + \bar v; \bar u^{(k)}) + \frac{1}{2\gamma_k} \|\bar v\|_2^2\right\},
\end{equation}
where $ \gamma_k$ is the step-size: the smaller $\gamma_k$ is, the closer the solution is to the current iterate.

The following proposition shows that the minimization step~\eqref{eq:min_model_step_gen} amounts to a linear quadratic risk-sensitive control problem. Prop.~\ref{prop:model_min} is then a sub-case of the following proposition. 
\begin{proposition}\label{prop:model_min_gen}
	The model minimization step~\eqref{eq:min_model_step_gen} is given as $\bar u^{(k+1)} = \bar u^{(k)} + \bar v^*$ where $\bar v^*$ is the solution of
	\begin{align}\label{eq:model_min_lin_quad_risk_ctrl_gen}
	\min_{\bar v \in \reals^{\horizon p}} \sup_{\bar w \in \reals^{\horizon q}  \bar y \in \reals^{\horizon d}} \quad & \sum_{t=1}^{\horizon}\left(\frac{1}{2}y_t^\top H_t y_t + \tilde h_t^\top y_t\right) + \sum_{t=0}^{\horizon-1} \left(\frac{1}{2}v_t^\top (G_t + \gamma_k^{-1}\id_p)v_t + \tilde g_t^\top v_t\right)  - \sum_{t=0}^{\horizon-1} \frac{1}{2\theta\sigma^2}\|w_t\|_2^2 \\
	\mbox{\textup{subject to}} \quad & y_{t+1} = A_t y_t + B_tv_t + C_t w_t \nonumber\\
	& y_0 = 0, \nonumber
	\end{align}
	where $x_t^{(k)} = \traj_t(\bar u^{(k)}, 0)$ and
	\begin{gather*}
	A_t = \nabla_x \dyn_t(x_t^{(k)}, u_t^{(k)}, 0)^\top\quad  B_t =  \nabla_u \dyn_t(x_t^{(k)}, u_t^{(k)}, 0)^\top\quad C_t = \nabla_{w} \dyn_t(x_t^{(k)}, u_t^{(k)}, 0)^\top \\ 
	H_t = \nabla^2 h_t(x_t^{(k)}) \quad \tilde h_t = \nabla h_t(x_t^{(k)}) \quad G_t = \nabla^2g_t(u_t^{(k)}) \quad \tilde g_t = \nabla g_t(u_t^{(k)}).
	\end{gather*} 
\end{proposition}
\begin{proof}
	To ease notations denote $\bar u^{(k)} = \bar u$.
	Recall that the trajectory defined by $\bar u, \bar w$ reads
	\begin{align*}
	\traj_1(\bar u, \bar w) & = \dyn_0(\hat x_0, F_0^\top \bar u, E_0^\top \bar w), &
	\traj_{t+1}(\bar u, \bar w) & = \dyn_t(\traj_t(\bar u, \bar w), F_t^\top \bar u, E_t^\top \bar w)
	\end{align*}
	where $F_t = e_{t+1} \otimes \id_p \in \reals^{\horizon p\times p} $ satisfies $F_t^\top \bar u = u_t$, $E_t = e_{t+1} \otimes \id_q \in \reals^{\horizon q\times q}$ satisfies $E_t^\top \bar w = w_t$ and $e_t\in \reals^\horizon$ is the $t$\textsuperscript{th} canonical vector in $\reals^\horizon$. The gradient is then given by
	\begin{align*}
	\nabla \traj_1 (\bar u, \bar w) & = \left( \begin{matrix}
	F_0 \nabla_u\dyn_0(\hat x_0, u_0, w_0) \\
	E_0 \nabla_w\dyn_0(\hat x_0, u_0, w_0)
	\end{matrix}\right) \\
	\nabla \traj_{t+1}(\bar u, \bar w) & = \nabla \traj_t(\bar u, \bar w) \nabla_x \dyn_t( \traj_t(\bar u, \bar w), u_t, w_t) + \left( \begin{matrix}
	F_t \nabla_u\dyn_t( \traj_t(\bar u, \bar w), u_t, w_t) \\
	E_t \nabla_w\dyn_t( \traj_t(\bar u, \bar w), u_t, w_t)
	\end{matrix}\right)
	\end{align*}
	For a given $\bar v = (v_0; \ldots; v_{\tau-1})$, the product $\bar y = (y_1; \ldots; y_{\horizon}) = \nabla \traj (\bar u, 0)^\top (\bar v, \bar w)$ reads
	\begin{align*}
	y_1 & = \nabla_u \dyn_0(x_0, u_0, 0)^\top v_0 + \nabla_w \dyn_0(x_0, u_0, 0)^\top w_0 \\
	y_{t+1} & = \nabla_x \dyn_t(x_t, u_t, 0)^\top y_t +\nabla_u \dyn_t (x_t, u_t, 0)^\top v_t + \nabla_w \dyn_t (x_t, u_t, 0)^\top w_t,
	\end{align*}
	where $x_t = \traj_t(\bar u, 0)$, $x_0= \hat x_0$  and we used that $y_t =  \nabla \traj_t(\bar u, 0) ^\top (\bar v, \bar w)$. 
	
	The approximate state objective inside the exponential in~\eqref{eq:model_gen} reads then
	\begin{align*}
	q_h\big(\bar x + \nabla_{\bar u} \traj(\bar u, 0)^\top \bar v + \nabla_{\bar w}\traj(\bar u, 0)^\top \bar w  ; \bar x  \big) = &
	\sum_{t=1}^{\horizon} q_{h_t}(x_t + y_t; x_t) 
	\label{eq:lin_opt_ctrl}\\
	 & \mbox{\textup{s.t.}} \quad y_{t+1} = A_t y_t + B_t v_t + C_t w_t\nonumber \\
	& \phantom{\mbox{\textup{s.t.}} \quad} y_0 = 0, \nonumber
	\end{align*}
	where $A_t = \nabla_x \dyn_t(x_t, u_t, 0)^\top, B_t =  \nabla_u \dyn_t(x_t, u_t, 0)^\top, C_t = \nabla_w \dyn_t(x_t, u_t, 0)^\top$. We retrieve the model of a linear quadratic control problem perturbed by noise $\bar w$. The risk sensitive objective can then be decomposed as in Proposition~\ref{prop:lin_quad_risk_ctrl_as_min_max}, leading to the claimed formulation.
\end{proof}

\subsection{ILEQG and RegILEQG implementations}
\subsubsection{Implementations by dynamic programming}
We present in Algo.~\ref{algo:RegILEQG_gen} the regularized variant of ILEQG that calls Algo.~\ref{algo:LEQG} at each step to solve the linear quadratic problem by dynamic programming. We present it for constant step-size. A variant with line-search could also be derived. We also present in Algo.~\ref{algo:ILEQG} the classical ILEQG method equipped with a line-search on the Monte-Carlo approximation of the objective.

\subsubsection{Implementation by automatic differentiation}
We consider here problems whose objective rely only in the last state, i.e.
\begin{equation}\label{eq:last_state_cost_gen}
h(\bar x) = h_\horizon(x_\horizon),
\end{equation}
and assume $h_\horizon$ strictly convex. In that case we can use automatic differentiation oracles as defined by~\citet{roulet2019iterative} and recalled below.
\begin{definition}[Automatic-differentiation oracle]\label{def:auto_diff} Let $\tilde x_\horizon: \reals^{\horizon \pi} \rightarrow \reals^d$ be a chain of compositions defined by 
	\[
	 x_0 = \hat x_0, \qquad x_{t+1} = \dyn(x_t, \omega_t) \qquad \mbox{for}\: t\in\{0, \ldots, \horizon-1\}
	\]
	for differentiable functions $\dyn_t: \reals^d \times \reals^\pi$, $\hat x_0 \in \reals^d$
	An \emph{automatic-differentiation oracle} is any procedure that computes $\nabla \traj_\horizon(\bar \omega)  z$ for  any $\bar \omega = (\omega_0, \ldots, \omega_{\horizon-1}) \in \reals^{\horizon \pi}$, $z \in \reals^{d}$.
\end{definition}
We can then use the dual optimization problem of~\eqref{eq:min_model_step_gen} as shown in the following proposition. For final-state cost~\eqref{eq:last_state_cost_gen}, the automatic differentiation implementation is computationally less expensive than a dynamic programming approach whose naive implementation requires the inversion of multiple matrices. The detailed implementation by automatic-differentiation oracle is provided in Algo.~\ref{algo:auto_diff}.
\begin{proposition}
	Consider the model minimization subproblem~\eqref{eq:min_model_step_gen} for strictly convex last state cost~\eqref{eq:last_state_cost_gen} and notations defined in Prop.~\ref{prop:model_min_gen}. If $\nabla^2h_\horizon(x_\horizon^{(k)})^{-1} \succ \theta\sigma^2 \nabla_{\bar w} \traj_\horizon(\bar u^{(k)}, 0)^\top \nabla_{\bar w} \traj_\horizon(\bar u^{(k)}, 0)$,  then 
	\begin{enumerate}[label=(\roman*)]
		\item the dual of subproblem~\eqref{eq:model_min_lin_quad_risk_ctrl_gen} reads
		\begin{equation}\label{eq:model_min_dual_gen}
		\min_{z \in \reals^d} \tilde q_{h_\horizon}^*(z) + \tilde q_g^*(-\nabla_{\bar u} \traj_\horizon(\bar u^{(k)}, 0)  z) - \frac{\theta\sigma^2}{2}\|\nabla_{\bar w} \traj_\horizon(\bar u^{(k)}, 0)  z\|_2^2,
		\end{equation}
		where $\tilde q_{h_\horizon}(y) = \frac{1}{2}y_\horizon^\top H_\horizon y_\horizon + \tilde h_\horizon^\top y_\horizon$, $ \tilde q_g(\bar v) = \frac{1}{2}\bar v^\top (\bar G + \stepsize_k^{-1} \id_{\horizon p})\bar v + \tilde g^\top\bar v$, $\bar G = \diag(G_0,\ldots, G_{\horizon-1})$, $\tilde g = (\tilde g_0, \ldots, \tilde g_{\horizon-1})$
		and for a function $f$, we denote by $f^*$ its convex conjugate,
		\item 	the	model minimization step is then given as $\bar u^{(k+1)} = \bar u^{(k)} +  \nabla \tilde q_g^*(-\nabla_{\bar u} \traj(\bar u^{(k)}, 0) z^*)$, where $z^*$ is solution of \eqref{eq:model_min_dual_gen},
		\item the model minimization step makes $10d +1$ calls to an automatic differentiation oracle defined in Def.~\ref{def:auto_diff} by using a conjugate gradient method to solve~\eqref{eq:model_min_dual_gen}.
	\end{enumerate}
\end{proposition}
\begin{proof}
	To ease notations denote $\bar u^{(k)} = \bar u$.
	Denoting $\tilde A = \nabla_{\bar u} \traj_\horizon(\bar u, 0)^\top$, $\tilde B = \nabla_{\bar w} \traj_\horizon(\bar u, 0)^\top$, $\tilde q_{h_\horizon}(y) = \frac{1}{2}y_\horizon^\top H_\horizon y_\horizon + \tilde h_\horizon^\top y_\horizon$, $ \tilde q_g(\bar v) = \frac{1}{2}\bar v^\top (\bar G + \stepsize_k^{-1} \id_{\horizon p})\bar v + \tilde g^\top\bar v$, $\bar G = \diag(G_0,\ldots, G_{\horizon-1})$, $\tilde g = (\tilde g_0, \ldots, \tilde g_{\horizon-1})$,  the model minimization subproblem~\eqref{eq:model_min_lin_quad_risk_ctrl_gen}  for last state cost~\eqref{eq:last_state_cost_gen} reads  
	\begin{align}
	& \min_{\bar v \in \reals^{\horizon p}}\sup_{\bar w \in \reals^{\horizon q}} \quad
	\tilde q_g(\bar v) + \tilde q_{h_\horizon}(\tilde A\bar v + \tilde B\bar w) - \frac{1}{2 \theta\sigma^2}\|\bar w\|_2^2  \nonumber\\
	= & \min_{\bar v \in \reals^{\horizon p}} \tilde q_g(\bar v) + \sup_{\bar w \in \reals^{\horizon q}}  \sup_{z \in \reals^d} z^\top (\tilde A\bar v + \tilde B\bar w) -\tilde q_{h_\horizon}^*(z) - \frac{1}{2 \theta\sigma^2}\|\bar w\|_2^2 \nonumber \\
	= & \min_{\bar v \in \reals^{\horizon p}}\sup_{z \in \reals^d} \tilde q_g(\bar v) +   z^\top \tilde A\bar v  -\tilde q_{h_\horizon}^*(z) + \frac{\theta\sigma^2}{2 }\|\tilde B^\top z \|_2^2. \label{eq:primal_mapping}
	\end{align}
	Recall that for a function $f(x) = x^\top q  + x^\top Q x/2$ with $Q\succ 0$, we have $f^*(z) = \sup_x \{z^\top x - f(x) \}= (z-q)^\top Q^{-1} (z-q)/2$.
	If $H_\horizon^{-1} \not \succeq \theta\sigma^2 \tilde B \tilde B^\top$ the supremum in $z$ is infinite.
	If $H_\horizon^{-1} \succ \theta\sigma^2 \tilde B \tilde B^\top$, the supremum in $z$ is finite. The problem is then a strongly convex-concave problem such that min and max can be inverted leading to the dual problem 
	\begin{align*}
	\max_{z \in \reals^d} - \tilde q_{h_\horizon}^*(z) -\tilde q_g^*(-\tilde A^\top z) + \frac{\theta\sigma^2}{2}\|\tilde B ^\top z\|_2^2.
	\end{align*}
	The primal solution is obtained from a dual solution $z^*$ by the mapping $ \bar v^* = \nabla \tilde q_g^*(-\tilde A^\top z^*)$ obtained from~\eqref{eq:primal_mapping}.
	
	The dual problem~\eqref{eq:model_min_dual_gen} is a quadratic problem, which can then be solved in $d$ iterations by a conjugate gradients method. The gradients of $z \rightarrow \tilde q_g^*(-\nabla_{\bar u} \traj(\bar u^{(k)}, 0) z)$ and $z \rightarrow \frac{\theta\sigma^2}{2}\|\nabla_{\bar w}\traj_\horizon(\bar u^{(k)}, 0) z\|_2^2$ can be computed by an automatic differentiation procedure defined in Def.~\ref{def:auto_diff}. Each gradient computation requires the equivalent of two calls to an automatic differentiation oracle as detailed by~\citet[Lemma 3.4]{roulet2019iterative}. The mapping to the primal solution costs an additional call. Finally, checking if the problem is feasible requires to compute the Hessian of $z \rightarrow \tilde q_{h_\horizon}^*(z) - \frac{\theta\sigma^2}{2}\|\tilde B ^\top z\|_2^2$ which costs $4d$ additional calls (each call computes the second order derivative with respect to a given coordinate in $\reals^d$ and computing the second order derivative amounts to back-propagate through the computation of the gradient of $z \rightarrow\frac{\theta\sigma^2}{2}\|\tilde B ^\top z\|_2^2$ which itself cost 2 calls to an automatic differentiation procedure).
\end{proof}
We detail the complete implementation by automatic differentiation in Algo.~\ref{algo:auto_diff}. We assume that we have access to a conjugate gradients method \texttt{conjgrad} for quadratic problems of the form 
\[
\min_{z \in \reals^d} \left\{f(z) := \frac{1}{2} z^\top A z + b^\top z\right\},
\]
with $A\succ 0$,
that given an oracle on the gradient of $f$ outputs the solution of the quadratic problem. Formally, it reads $\texttt{conjgrad}(\nabla f) = \argmin_{z\in \reals d} f(z)$. This can be implemented following~\citet[Section 1.3.2.]{nesterov2013introductory}.
Finally note that the leading dimension of the problem is the length $\horizon$ of the dynamics. By expressing the complexity in terms of automatic differentiation oracle, we capture the main complexity of the algorithm. We ignore in particular the cost of inverting the Hessian of the final state objective and the cost of checking if the subproblems are positive definite.

%% file: appendix/03b_algos_code.tex
\begin{algorithm}
	\caption{Dynamic programming for Linear Exponential Quadratic Gaussian (LEQG)~\eqref{eq:min_max_ctrl} \label{algo:LEQG}}
	\begin{algorithmic}[1]
		\State{{\bfseries Inputs:} Initial state $\hat x_0$, risk-sensitivity parameter $\theta$, variance $\sigma^2$, convex quadratic costs $H_t \succeq 0, \tilde h_t$, strictly convex quadratic costs $G_t \succ 0, \tilde g_t$, linear dynamics $A_t, B_t, C_t$}	
		\State{\underline{\textit{Backward pass}}}
		\State{Initialize $P_\horizon = H_\horizon$, $p_\horizon = \tilde h_\horizon$, $\texttt{feasible}=\mbox{True}$}
		\For{$t=\horizon-1, \ldots, 0$}
		\If{$(\theta\sigma^2)^{-1} > \lambda_{\max}(C_t^\top P_{t+1}C_t)$ \label{state:dyn_prog_cond}}
		\State{\label{state:dyn_prog_rec}Compute 
			\begin{gather}
			\tilde P_{t+1} = P_{t+1} + P_{t+1}C_t((\theta\sigma^2)^{-1} \id_q - C_t^\top  P_{t+1}C_t)^{-1}C_t^\top P_{t+1} \\
			\tilde p_{t+1}  = p_{t+1} + P_{t+1}C_t((\theta\sigma^2)^{-1} \id_q - C_t^\top  P_{t+1}C_t)^{-1}C_t^\top p_{t+1} \\
			P_t  = H_t + A_t^\top \tilde P_{t+1}A_t - A_t^\top\tilde P_{t+1}B_t (G_t + B_t^\top\tilde P_{t+1}B_t)^{-1} B_t^\top \tilde P_{t+1} A_t \\
			p_t = \tilde h_t + A_t^\top \big[\tilde p_{t+1} - \tilde P_{t+1}B_t (G_t + B_t^\top\tilde P_{t+1}B_t)^{-1} [B_t^\top \tilde p_{t+1} +\tilde g_t] \big]
			\end{gather}}
		\State{\label{state:gain_matrix}Store
			\begin{align*}
			K_t &= -(G_t + B_t^\top\tilde P_{t+1}B_t)^{-1} B_t^\top \tilde P_{t+1} A_t & 
			L_t^x &= ((\theta\sigma^2)^{-1} \id_q - C_t^\top P_{t+1}C_t)^{-1}C_t^\top P_{t+1}A_t \\
			k_t &= -(G_t + B_t^\top\tilde P_{t+1}B_t)^{-1}(\tilde g_t + B_t^\top \tilde p_{t+1}) & L_t^u &= ((\theta\sigma^2)^{-1} \id_q - C_t^\top P_{t+1}C_t)^{-1}C_t^\top P_{t+1}B_t\\
			& & 	l_t &= ((\theta\sigma^2)^{-1} \id_q - C_t^\top P_{t+1}C_t)^{-1}C_t^\top p_{t+1}
			\end{align*}
		}
		\Else 
		\State{State $\texttt{feasible}=\mbox{False}$}
		\State{{\bfseries break}}
		\EndIf
		\EndFor
		\State{\underline{\textit{Roll-out pass}}}
		\If{\texttt{feasible}}
		\State{Initialize $x_0 = \hat x_0$}
		\For{$t=0,\ldots,\horizon-1$}
		\State{Compute 
			\begin{gather}
			u^*_t = K_t x_t + k_t \qquad w_t^* = L_t^x x_t + L_t^u u_t^* + l_t \\
			x_{t+1} = A_t x_t + B_t u^*_t + C_t w_t^*
			\end{gather}
		}
		\EndFor
		\Else
		\State{$ u_t^* = \mbox{None}$ for all $t$}
		\EndIf
		\State{{\bfseries Output:} $\bar u^*=(u^*_0; \ldots ; u_{\horizon-1}^*)$}
	\end{algorithmic}
\end{algorithm}

\begin{algorithm}
	\caption{Regularized Iterative Linear Exponential Quadratic Gaussian  \label{algo:RegILEQG_gen} (RegILEQG) \eqref{eq:min_model_step}}
	\begin{algorithmic}[1]
		\State{{\bfseries Inputs:} Initial state $\hat x_0$, risk sensitive parameter $\theta$,  variance $\sigma^2$, fixed step-size $\gamma$, initial command $\bar u^{(0)}$, number of iterations $K$, convex costs $h_t$, $g_t$, dynamics $\dyn_t$ }
		\For{$k=0,\ldots, K$}
		\State{\underline{ \textit{Forward pass}}}
		\State{Compute  along the noiseless trajectory $\bar x^{(k)} = \traj(\bar u^{(k)},0)$ defined by $\bar u^{(k)}$,
			\begin{gather*}
			H_t = \nabla^2 h_t(x_t^{(k)}) \quad \tilde h_t = \nabla h_t(x_t^{(k)}) \quad G_t = \nabla^2 g_t(u_t^{(k)}) \quad \tilde g_t = \nabla g_t(u_t^{(k)}) \\
			A_t = \nabla_x \dyn_t(x_t^{(k)}, u_t^{(k)}, 0)^\top \quad B_t =  \nabla_u \dyn_t(x_t^{(k)}, u_t^{(k)}, 0)^\top \quad C_t = \nabla_w \dyn_t(x_t^{(k)}, u_t^{(k)}, 0)^\top
			\end{gather*}
		}
		\State{\underline{ \textit{Backward pass}}}
		\State{Apply Algo.~\ref{algo:LEQG} to \label{line:RegILEQG_step}
			\begin{align}
			\min_{\bar v \in \reals^{\horizon p}} \sup_{\bar w \in \reals^{\horizon d}} \quad & \sum_{t=1}^{\horizon}\left(\frac{1}{2}y_t^\top H_t y_t + \tilde h_t^\top y_t\right) + \sum_{t=0}^{\horizon-1} \left(\frac{1}{2}v_t^\top (G_t + \gamma^{-1}\id_p)v_t + \tilde g_t^\top v_t\right) - \sum_{t=0}^{\horizon-1} \frac{1}{2\theta\sigma^2}\|w_t\|_2^2 \nonumber \\
			\mbox{\textup{subject to}} \quad & y_{t+1} = A_t y_t +B_tv_t + C_t w_t \nonumber\\
			& y_0 = 0\nonumber
			\end{align}
		}
		\If{Algo.~\ref{algo:LEQG} cannot output a solution}
		\State{State $\texttt{feasible} = \mbox{False}$ }
		\State{{\bfseries break}}
		\Else
		\State{Update $\bar u^{(k+1)} = \bar u^{(k)} + \bar v^*$, with $\bar v^*$ found in Step~\ref{line:RegILEQG_step}}
		\EndIf
		\EndFor
		\State{{\bfseries Output:}  $\bar u^{(K)}$ if $\texttt{feasible}$  or last iterate $\bar u^{(k)}$ if not $\texttt{feasible}$}
	\end{algorithmic}
\end{algorithm}

\begin{algorithm}
	\caption{Iterative Linear Exponential Quadratic Gaussian (ILEQG) \label{algo:ILEQG} \eqref{eq:ileqg}}
	\begin{algorithmic}[1]
		\State{{\bfseries Inputs:} Initial state $\hat x_0$, risk sensitive parameter $\theta$, variance $\sigma^2$,  initial command $\bar u^{(0)}$, number of iterations $K$, convex costs $h_t$, $g_t$, dynamics $\dyn_t$, line-search precision $\epsilon$ }
		\For{$k=0,\ldots, K$}
		\State{\underline{\textit{Forward pass}}}
		\State{Compute  along the noiseless trajectory $\bar x^{(k)} = \traj(\bar u^{(k)},0)$ defined by $\bar u^{(k)}$,
			\begin{gather*}
			H_t = \nabla^2 h_t(x_t^{(k)}) \quad \tilde h_t = \nabla h_t(x_t^{(k)}) \quad G_t = \nabla^2 g_t(u_t^{(k)}) \quad \tilde g_t = \nabla g_t(u_t^{(k)}) \\
			A_t = \nabla_x \dyn_t(x_t^{(k)}, u_t^{(k)}, 0)^\top \quad B_t =  \nabla_u \dyn_t(x_t^{(k)}, u_t^{(k)}, 0)^\top \quad C_t = \nabla_w \dyn_t(x_t^{(k)}, u_t^{(k)}, 0)^\top
			\end{gather*}
		}
		\State{\underline{\textit{Backward pass}}}
		\State{Apply Algo.~\ref{algo:LEQG} to \label{line:ILEQG_step}
			\begin{align}
			\min_{\bar v \in \reals^{\horizon p}} \sup_{\bar w \in \reals^{\horizon d}} \quad & \sum_{t=1}^{\horizon}\left(\frac{1}{2}y_t^\top H_t y_t + \tilde h_t^\top y_t\right) + \sum_{t=0}^{\horizon-1} \left(\frac{1}{2}v_t^\top G_tv_t + \tilde g_t^\top v_t\right) - \sum_{t=0}^{\horizon-1} \frac{1}{2\theta \sigma^2}\|w_t\|_2^2 \nonumber \\
			\mbox{\textup{subject to}} \quad & y_{t+1} = A_t y_t +B_tv_t + C_t w_t \nonumber\\
			& y_0 = 0 \nonumber
			\end{align}
		}
		\If{Algo.~\ref{algo:LEQG} cannot output a solution}
		\State{State $\texttt{feasible} = \mbox{False}$ }
		\State{{\bfseries break}}
		\Else
		\State{Find $\alpha>0$ such that $\bar u^{(k+1)} = \bar u^{(k)} + \alpha \bar v^*$, with $\bar v^*$ found in Step~\ref{line:ILEQG_step}, satisfies 
			\[
			\tilde f_\theta(\bar u^{(k+1)}) \leq \tilde f_\theta(\bar u^k) + \epsilon
			\]
			\hspace{3em}where $\tilde f_\theta(\bar u)$ is the Monte-Carlo approximation of the risk-sensitive cost
		}
		\EndIf
		\EndFor
		\State{{\bfseries Output:}  $\bar u^{(K)}$ if $\texttt{feasible}$  or last iterate $\bar u^{(k)}$ if not $\texttt{feasible}$}
	\end{algorithmic}
\end{algorithm}

\begin{algorithm}
	\begin{algorithmic}[1]
		\caption{Regularized Iterative Linear Exponential Gaussian (RegILEQG) \eqref{eq:min_model_step} \\
			\hspace*{52pt} 	using automatic differentiation oracles\label{algo:auto_diff}  
		for final state cost~\eqref{eq:last_state_cost_gen}}
		\State{{\bfseries Inputs:} Initial state $\hat x_0$, risk sensitive parameter $\theta$, variance $\sigma^2$, step-size $\gamma$, initial command $\bar u^{(0)}$, number of iterations $K$, convex costs $g_t$, final strictly convex cost $h_\horizon$, dynamics $\dyn_t$}
		\For{$k=0,\ldots, K$}
		\State{\underline{\textit{Forward pass}}}
		\State{Compute $\bar x^{(k)} = \traj(\bar u^{(k)}, 0)$ along the trajectory}
		\State{Store $\nabla \dyn_t(x_t^{(k)}, u_t^{(k)}, 0))$ to compute any 
		$\nabla_{\bar u} \traj(\bar u^{(k)}, 0) z$ or $\nabla_{\bar w} \traj(\bar u^{(k)}, 0) z$ by automatic-differentiation}
		\State{\underline{\textit{Dual formulation}}}
		\State{Compute $H_\horizon = \nabla^2 h_\horizon(\bar x_\horizon^{(k)}), h_\horizon = \nabla h(\bar x_\horizon^{(k)})$, $G_t = \nabla^2 g_t(u_t^{(k)})$, $\tilde g_t = \nabla g_t(u_t^{(k)})$}
		\State{Define $\tilde q_{h_\horizon}^*: z \rightarrow \frac{1}{2}(z-\tilde h_\horizon)^\top H_\horizon^{-1}(z-\tilde h_\horizon)$ }
		\State{Define $\tilde q_{g}^*: \bar \zeta \rightarrow \frac{1}{2}(\bar \zeta -\tilde g)^\top (\bar G + \stepsize_k^{-1}\id_{\horizon p})(\bar \zeta -\tilde g)$ where $\bar G = \diag(G_0, \ldots, G_{\horizon -1})$, $\tilde g = (g_0;\ldots;g_{\horizon-1})$
		}
		\State{Define $\nabla \tilde q_{g}^*: \bar \zeta \rightarrow (\bar G + \stepsize_k^{-1}\id_{\horizon p})(\bar \zeta-\tilde g)$
		}
		\State{Define 
			\[
			f:  z \rightarrow \tilde q_{h_\horizon}^*(z) + \tilde q_{g}^*(- \nabla_{\bar u} \traj_\horizon(\bar u^{(k)}, 0)z) -  \frac{\theta\sigma^2}{2}\|\nabla_{\bar w} \traj_\horizon(\bar u^{(k)}, 0)z\|_2^2
			\]
		\hspace{3ex}	where $\nabla_{\bar u} \traj_\horizon(\bar u^{(k)}, 0)z$ and $\nabla_{\bar w} \traj_\horizon(\bar u^{(k)}, 0)z$ are computed by automatic differentiation
		}
		\State{\underline{\textit{Update pass}}}
		\State{Define $r: z \rightarrow q_{h_\horizon}^*(z) -  \frac{\theta\sigma^2}{2}\|\nabla_{\bar w} \traj_\horizon(\bar u^{(k)}, 0)z\|_2^2$}
		\State{Compute $\nabla^2r(z)$ for e.g. $z=0$}
		\If{$\nabla^2r(z) \not \succ 0$}
		\State{State $\texttt{feasible} = \mbox{False}$}
		\State{{\bfseries break}}
		\Else
		\State{Compute $z^* = \texttt{conjgrad}(\nabla f) =\argmin_{z\in \reals d} f(z)$ where $\nabla f$ is provided by automatic differentiation}
		\State{Map to primal solution $\bar u^{(k+1)}= \bar u^{(k)} +  \nabla \tilde q_g^*(-\nabla_{\bar u} \traj(\bar u^{(k)}, 0)  z^*)$}
		\EndIf
		\EndFor
		\State{{\bfseries Output:}  $\bar u^{(K)}$ or last iterate $\bar u^{(k)}$ if not $\texttt{feasible}$}
	\end{algorithmic}
\end{algorithm}

%% file: appendix/04_cvg_proofs.tex
\subsection{Gradient of the risk-sensitive objective}
We recall the derivation of the gradient a risk-sensitive objective below. The proof follows from standard derivations.
\begin{proposition}\label{prop:grad_comput}
	Given a differentiable function $f: \reals^{\horizon p + \horizon q} \rightarrow \reals$, define 
	\[
	F: \bar u \rightarrow \frac{1}{\theta} \log \Expect_{\bar w \sim \mathcal{N}(0, \sigma ^2\id_{\horizon q})} \exp(\theta f(\bar u, \bar w)).
	\]
	Then for $\bar u \in \reals^{\horizon p}$ such that $F(\bar u)< +\infty$,
	\[
	\nabla F(\bar u) = \frac{\Expect_{\bar w \sim \mathcal{N}(0, \sigma^2\id_{\horizon q})} \exp(\theta f(\bar u, \bar w)) \nabla_{\bar u} f(\bar u, \bar w)}{ \Expect_{\bar w \sim \mathcal{N}(0, \sigma^2\id_{\horizon q})} \exp(\theta f(\bar u, \bar w))} = \Expect_{\bar w \sim p(\cdot;\bar u)} \nabla_{\bar u} f(\bar u, \bar w),
	\] 
	where 
	\begin{align*}
		p(\bar w; \bar u) & = \exp\left(\theta f(\bar u,\bar w) -\frac{1}{2\sigma^2} \|\bar w\|_2^2 - \theta F(\bar u)\right).
	\end{align*}
\end{proposition}

\subsection{Surrogate risk-sensitive objective}
We study the surrogate risk-sensitive objective, its truncated gradient and the link with ILEQG in the following propositions. We present them for the quadratic case where we use extensively that the second order Taylor expansion of a quadratic is equal to itself. Formally, for a quadratic $q$, we have for any $x,y$ that $q(x+y) = q(x) + \nabla q(x)^\top y + \frac{1}{2}y^\top \nabla^2 q(x) y$ and $\nabla q(x+y) = \nabla q(x) + \nabla^2 q(x) y$, i.e., that the gradient is an affine function. Recall that we denote by $\tilde x(\bar u)$ the trajectory induced by the control $\bar u$ as defined in~\eqref{eq:traj_def}.

\approxrisk* 
\begin{proof}
	For $\bar u \in \reals^{\horizon p}$, since $h$ is quadratic and $\bar w \rightarrow \theta h(\traj(\bar u) + \nabla \traj(\bar u)^\top \bar w) - \|\bar w\|_2^2/2\sigma^2 $ is strongly concave, the function $p(\cdot;\bar u)$ is the density of a Gaussian where $\theta \hat \eta(\bar u)$ is its log-partition function.
	It can be factorized as follows using $h(\bar x + \bar y)   =   h(\bar x) +  \nabla h(\bar x)^\top \bar y + \frac{1}{2} \bar y^\top \nabla^2 h(\bar x) \bar y $ and denoting $X = \nabla \traj(\bar u)$, $ \tilde h = \nabla h(\bar x) , H = \nabla^2 h (\bar x) $, $\bar x  = \traj(\bar u)$, 
	\begin{align}
	\theta h(\bar x + \nabla \traj(\bar u)^\top \bar w ) - \frac{1}{2\sigma^2}\|\bar w\|_2^2 & =  \theta h(\bar x) +  \theta (X \tilde h)^\top \bar w + \frac{\theta}{2}\bar w^\top X H  X^\top \bar w - \frac{1}{2\sigma^2}\|\bar w\|_2^2 \nonumber\\
	& =   \theta h(\bar x) -\frac{1}{2}(\bar w - \bar w_*)^\top \Sigma^{-1}(\bar w-\bar w_*)   + \frac{1}{2} \bar w_*^\top \Sigma^{-1} \bar w_*\label{eq:gaussian_fact}
	\end{align}
	where $ \Sigma^{-1} = (\sigma^{-2}\id_{\horizon p} - \theta X HX^\top)\succ 0$  and
	\begin{align*}
	\bar w_* 	= \argmax_{\bar w \in \reals^{\horizon p}} \left\{\theta(X \tilde h)^\top \bar w - \frac{1}{2}\bar w^\top(\sigma^{-2}\id_{\horizon p} - \theta X HX^\top)\bar w\right\} = \theta(\sigma^{-2}\id_{\horizon p} - \theta X HX^\top)^{-1} X \tilde h.
	\end{align*}
	The claim follows from the factorization in~\eqref{eq:gaussian_fact}. 
		The surrogate risk-sensitive cost can then be computed analytically and reads
		\begin{align*}
		\hat \eta(\bar u) & = \frac{1}{\theta}\log \int (2\pi\sigma^2)^{-\tau p/2}\exp\left[\theta h(\traj(\bar u)  +\nabla \traj(\bar u)^\top \bar w) -\frac{1}{2\sigma^2} \|\bar w\|_2^2\right]d\bar w\\
		& =  \frac{1}{\theta} \log\left( \sqrt{\det(\sigma^{-2}\Sigma)}\exp\left[\theta h(\bar x) + \frac{1}{2} \bar w_*^\top \Sigma^{-1} \bar w_*\right] \right) \\
 &  = - \frac{1}{2\theta}\log \det(\id_{\horizon p} - \theta \sigma^2  X HX^\top) + h(\bar x)  + \frac{\theta\sigma^2}{2}\tilde h^\top X^\top(\id_{\horizon p} - \theta\sigma^2 X HX^\top)^{-1} X \tilde h.
		\end{align*}
\end{proof}
As a corollary we get an expression for the truncated gradient.
\begin{corollary}\label{corr:grad_approx_comput}
	Given $\bar u \in \reals^{\horizon p}$ such that condition~\eqref{eq:approx_risk_cond} holds, the  truncated gradient of  the surrogate risk sensitive cost reads
	\begin{align*}
	\widehat \nabla \hat \eta_\theta(\bar u) & = \nabla \traj(\bar u)
	\nabla h(\traj( \bar u) +\nabla \traj (\bar u)^\top \bar w_*)
	\end{align*}
	where 	$\bar w_* $ is given in~\eqref{eq:gaussian_def}.
\end{corollary}
\begin{proof}
For any affine function of the variable $\bar w$ we have $\Expect_{\bar w \sim \hat p(\cdot; \bar u)} [ A \bar w + b] = A\bar w_* + b$. Since the truncated gradient is the mean of an affine function of $\bar w$ we get the result.
\end{proof}

We can then link the truncated gradient to the RegILEQG step.
\RegILEQGsteptruncgrad*
\begin{proof}
	To ease notations denote $\bar u^{(k)} = \bar u$, $\bar u^{(k+1)} = \bar u^+$ and $\stepsize_k = \stepsize$ such that the RegILEQG step reads $\bar u^+ = \bar u + \bar v^*$ where $\bar v^*$ is the solution of the min-max problem in~\eqref{eq:model_min_lin_quad_risk_ctrl}
	\begin{align*}
	\min_{\bar v \in \reals^{\horizon p}}\max_{\bar w \in \reals^{\horizon p}} q_h(\bar x + \nabla \traj(\bar u)^\top(\bar  v + \bar w); \bar x) + q_g(\bar u + \bar v; \bar u) + \frac{1}{2\stepsize} \|\bar v\|_2^2 - \frac{1}{2\theta\sigma^2} \|\bar w\|_2^2
	\end{align*}
	where $\bar x = \traj(\bar u)$, $q_h(\bar x + \bar y; \bar x) = h(\bar x + \bar y) = h(\bar x) + \nabla h(\bar x)^\top \bar y + \frac{1}{2} \bar y^\top \nabla^2 h(\bar x) \bar y$, same for $q_g$.
	Denote $\tilde g = \nabla g(\bar u) , G = \nabla^2g (\bar u), \tilde h = \nabla h(\bar x) , H = \nabla^2 h (\bar x) $ and $X = \nabla \traj(\bar u)$. The problem is then equivalent to
	\begin{align}
	& \min_{\bar v \in \reals^{\horizon p}}  (\tilde g + X \tilde h)^\top \bar v + \frac{1}{2} \bar v^\top (G + \stepsize^{-1}\id_{\horizon p} + X HX^\top)\bar v + \max_{\bar w \in \reals^{\horizon p}} (X \tilde h + X HX^\top \bar v)^\top \bar w - \frac{1}{2}\bar w^\top((\theta\sigma^2)^{-1}\id_{\horizon p} - X HX^\top)\bar w \nonumber\\
	= & \min_{\bar v \in \reals^{\horizon p}}  (\tilde g + X \tilde h)^\top \bar v + \frac{1}{2} \bar v^\top (G + \stepsize^{-1}\id_{\horizon p}+ X HX^\top)\bar v + \frac{1}{2}(X \tilde h + X HX^\top \bar v)^\top((\theta\sigma^2)^{-1}\id_{\horizon p} - X HX^\top)^{-1} (X \tilde h + X HX^\top \bar v) \label{eq:RegILEQG_comput}
	\end{align}
	where we used $(\sigma^{-2}\id_{\horizon p} - \theta X HX^\top) \succ 0 $ by assumption. The objective in~\eqref{eq:RegILEQG_comput} is the model $m_{f_\theta}$ expressed as a function of $\bar v$ and is clearly convex. 
	Denote 
	\begin{align*}
	 \bar w_* &  = ((\theta\sigma^2)^{-1}\id_{\horizon p} - X HX^\top)^{-1} X \tilde h
	\end{align*}
	which is equal to $\bar w_*$ defined in Prop.~\ref{prop:risk_approx_comput}.
	The solution of the problem reads then
	\begin{align*}
	\bar v^* = -(G+ \stepsize^{-1}\id_{\horizon p}+ R )^{-1}(\tilde g + X\tilde h +  X HX^\top \bar w^* )
	\end{align*}
	where 
	\begin{align*}
	R &  = XHX^\top + XHX^\top((\theta\sigma^2)^{-1}\id_{\horizon p} - XHX^\top)^{-1}XHX^\top
	\end{align*}
	The truncated gradient from Corr.~\ref{corr:grad_approx_comput} reads
	\begin{align*}
	\widehat \nabla \hat \eta_\theta(\bar u) & = \nabla \traj(\bar u)
	\nabla h(\traj( \bar u) +\nabla \traj (\bar u)^\top \bar w_*) \\
	& = X(\tilde h + HX^\top  \bar w_*)
	\end{align*}
	which concludes the proof.
\end{proof}
\paragraph{Extensions to non-quadratic case}
Prop.~\ref{prop:risk_approx_comput},~\ref{prop:ILEQG_trunc_grad} and Corr.~\ref{corr:grad_approx_comput} also hold for non-quadratic costs by considering 
\[
\tilde \eta_\theta(\bar u) =  \frac{1}{\theta}\log\Expect_{\bar w } \exp[\theta q_h(\traj(\bar u) + \nabla \traj(\bar u)^\top \bar  w; \traj(\bar u))].
\]
in place of $\hat \eta_\theta$ and
\[
\widetilde \nabla \tilde \eta_\theta(\bar u) = \Expect_{\bar w \sim \tilde p_{(\cdot; \bar u)}} \nabla \traj(\bar u)\nabla q_h(\traj(\bar u) + \nabla \traj(\bar u)^\top \bar w; \traj(\bar u))
\]
in place of $\widehat \nabla \hat \eta_\theta(\bar u)$ where
\[
\tilde p(\bar w; \bar u) = \exp\left(\theta q_h(\traj(\bar u) + \nabla \traj(\bar u)^\top \bar  w; \traj(\bar u)) -\frac{1}{2\sigma^2}\|\bar w\|_2^2 - \theta\tilde \eta_{\theta}(\bar u)\right)
\]
Precisely, the surrogate risk-sensitive cost $ \tilde \eta_\theta(\bar u)$ is defined if condition~\eqref{eq:approx_risk_cond} holds, the probability distribution $\tilde p$ is given by the same Gaussian and the expression of the surrogate is the same.  Prop.~\ref{prop:ILEQG_trunc_grad} is valid by replacing $\widehat \nabla \hat \eta_\theta(\bar u)$ by $\widetilde \nabla \tilde \eta_\theta(\bar u)$.

\subsection{Convergence analysis}
Recall the assumptions made for the convergence analysis.
\asscvg*
On $\mathcal{X}= \tilde x (\reals^{\horizon p})$, $h$ is Lipschitz continuous, denote $\ell_h(\mathcal{X})$ the Lipschitz parameter. Using that $h(\bar x) = \frac{1}{2}(\bar x- \bar x^*)^\top H (x-x^*) + \min_{\bar x} h(\bar x)$ with $H =\nabla^2 h( \bar x)$ and $\bar x^* \in \argmin_{\bar x}h(\bar x)$,  we get $\|\nabla h(\bar x) \|_2 \leq L_h\|\bar x-\bar x^*\|_2$ and so 
\begin{equation}\label{eq:lip_h}
\ell_h(\mathcal{X}) \leq L_h M_{\traj}
\end{equation}

We detail the approximation made by the truncated gradient in the following proposition.
\begin{proposition}\label{prop:trunc_approx}
	Under Asm.~\ref{ass:cvg}, we have for any $\bar u \in \reals^{\horizon p}$,
	\[
	\|\nabla \hat \eta_\theta(\bar u)- \widehat\nabla \hat \eta_\theta(\bar u)\|_2 \leq 
	\theta \tilde \sigma^2 L_h^2  L_{\tilde x}   \ell_{\traj} M_{\traj}^2+ \theta^2 \tilde \sigma^4 L_h^3L_{\traj}\ell_{\traj}^3 M_{\traj}^2 + \tau p \tilde \sigma^2 L_h L_{\traj}  \ell_{\traj}.
	\]
\end{proposition}
\begin{proof}
	We have with $\hat p(\cdot;\bar u)$ defined in~\eqref{eq:gaussian_prob_def}, and denoting $ \tilde h = \nabla h(\bar x) , H = \nabla^2 h (\bar x) $ and $X = \nabla \traj(\bar u)$ for $\bar x= \tilde x(\bar u)$,
	\begin{align}
		\nabla \hat \eta_\theta(\bar u)- \widehat\nabla \hat \eta_\theta(\bar u) 
		& = \Expect_{\bar w \sim \hat p(\cdot; \bar u)} \nabla^2 \traj(\bar u)[\cdot, \bar w, \nabla h(\traj( \bar u) + \nabla \traj (\bar u)^\top \bar w)] \nonumber\\
		& = \Expect_{\bar w \sim \hat p(\cdot; \bar u)}\left[ \nabla^2 \tilde x(\bar u)[\cdot, \bar w, \tilde h] + \nabla^2 \tilde x(\bar u)[\cdot, \bar w, HX^\top \bar w] \right] \\
		& = \nabla^2 \tilde x[\cdot, \bar w_*, \tilde h] + \left(\begin{matrix}
		\Tr(\mathcal{X}_{1, \cdot, \cdot} HX^\top\Expect_{\bar w \sim \hat p(\cdot; \bar u)}[\bar w\bar w^\top] ) \\
		\vdots \\
		\Tr(\mathcal{X}_{\tau p, \cdot, \cdot} HX^\top\Expect_{\bar w \sim \hat p(\cdot; \bar u)}[\bar w\bar w^\top] ),
		\end{matrix}\right)
	\end{align}
where  $\mathcal{X} = \nabla^2 \tilde x(\bar u)$ and we used the notations defined in Appendix~\ref{sec:notations}.
	We have then 
	\begin{align*}
		\Expect_{\bar w \sim \hat p(\cdot; \bar u)}[\bar w\bar w^\top] & = \operatorname{\mathbb{V}ar}_{\bar w\sim \hat p(\cdot;\bar u)}(\bar w) + \Expect_{\bar w \sim \hat p(\cdot;\bar u)}(\bar w)\Expect_{\bar w \sim \hat p(\cdot;\bar u)}(\bar w)^\top = \Sigma + \bar w_* \bar w_* ^\top 
	\end{align*}
		where $\bar w_*$ and $\Sigma$ are defined in~\eqref{eq:gaussian_def}.
So we get 
\begin{align*}
\nabla \hat \eta_\theta(\bar u)- \widehat\nabla \hat \eta_\theta(\bar u)  & = \nabla^2 \tilde x[\cdot; \bar w_*, \tilde h] + \nabla^2 \tilde x(\bar u)[\cdot, \bar w_*, HX^\top\bar w_*] + \sum_{i=1}^{\horizon p} \nabla^2 \tilde x(\bar u)[\cdot, u_i, HX^\top u_i]
\end{align*}
where $\Sigma = \sum_{i=1}^{\horizon p} u_iu_i^\top$ with $\|u_i\|_2^2 \leq \lambda_{\max}(\Sigma)$.
Therefore
\begin{align*}
	\|\nabla \hat \eta_\theta(\bar u)- \widehat\nabla \hat \eta_\theta(\bar u)\|_2&  \leq L_{\tilde x}\|\bar w_*\|_2\ell_{h}(\mathcal{X}) + L_{\traj} \|\bar w_*\|^2_2L_h\ell_{\traj} + \tau p L_{\traj} \|\Sigma\|_2L_h\ell_{\traj}
\end{align*}
	where $\ell_h(\mathcal{X})$ is the Lipschitz parameter of $h$ on $\mathcal{X}= \tilde x (\reals^{\horizon p})$ that can be bounded by~\eqref{eq:lip_h} and we used the tensor norm defined in~\eqref{eq:tensor_norm}. The bound follows, using the definitions of $\bar w_*$ and $\Sigma$, i.e.,
	\begin{align*}
		\|\bar w_*\|_2 & \leq \theta (\sigma^{-2}-\theta L_h\ell_{\traj}^2)^{-1}\ell_{\traj}\ell_{h}(\mathcal{X}), \\
		\|\Sigma\|_2 & \leq (\sigma^{-2}-\theta L_h\ell_{\traj}^2)^{-1}.
	\end{align*}
\end{proof}
The convergence under appropriate sufficient decrease condition is presented in the following proposition.
\cvg*
\begin{proof}
	Under Ass.~\ref{ass:cvg}, the model $m_{f_\theta}( \bar v; \bar u^{(k)})$ defined in~\eqref{eq:model} is well-defined and convex as shown for example in Prop.~\ref{prop:ILEQG_trunc_grad}. By using that $\bar v \rightarrow 	m_{f_\theta}( \bar v; \bar u^{(k)}) +\frac{1}{2\gamma_k} \|\bar v -\bar u^{(k)}\|_2^2$ is $\gamma_k^{-1}$ strongly convex with minimum achieved on $\bar u_{k+1}$ we get
	\begin{align}
		\hat f_\theta(\bar u^{(k)}) = m_{f_\theta}(\bar u^{(k)} ; \bar u^{(k)}) & \geq m_{f_\theta}(\bar u^{(k+1)}; \bar u^{(k)}) + \frac{1}{\gamma_k} \|\bar u^{(k+1)} -\bar u^{(k)}\|_2^2  \nonumber\\
		& \stackrel{\eqref{eq:suff_decrease}}{\geq} \hat f_\theta(\bar u^{(k+1)}) +   \frac{1}{2\gamma_k} \|\bar u^{(k+1)} -\bar u^{(k)}\|_2^2. \label{eq:suff_decrease_app}
	\end{align}
	Rearranging the terms and summing the inequalities we get 
	\begin{align*}
		\frac{1}{K}\sum_{k=0}^{K-1} \frac{1}{2\gamma_k} \|\bar u^{(k+1)} -\bar u^{(k)}\|_2^2 \leq \frac{\hat f_\theta(\bar u^{(0)}) - \hat f_\theta(\bar u^{(K)})}{K}.
	\end{align*}
	Now using Proposition~\ref{prop:ILEQG_trunc_grad}, we have that 
	\[
	\|\nabla g(\bar u^{(k)}) + \widehat \nabla \hat \eta_\theta(\bar u^{(k)})\|_2 \leq (L_g + \gamma^{-1} + \|R\|_2)	\|\bar u^{(k+1)} -\bar u^{(k)}\|_2,
	\]
	where 
	\begin{align*}
		\|R\|_2 &  = \|XH^{\nicefrac{1}{2}}(\id - H^{\nicefrac{1}{2}}X^\top(XHX^\top -(\theta\sigma^2)^{-1}\id)^{-1}XH^{\nicefrac{1}{2}}) H^{\nicefrac{1}{2}}X^\top\|_2\\
		& = \|XH^{\nicefrac{1}{2}}(\id - \theta\sigma^2 H^{\nicefrac{1}{2}} XX^\top H^{\nicefrac{1}{2}})^{-1}H^{\nicefrac{1}{2}}X^\top\|_2 \\
		& \leq \frac{\ell_{\traj}^2L_h}{1 - \theta\sigma^2 \ell_{\traj}^2 L_h },
	\end{align*}
	using that for a semi-definite positive matrix $A$ s.t $0\preceq A \prec \id$, $\|I-A\|_2 \geq 1- \lambda_{\max}(A)$ and $\|H^{\nicefrac{1}{2}}\|_2^2 = \|H\|_2$.
	Therefore we get
	\begin{align*}
		\min_{k =0,\ldots, K-1}\|\nabla g(\bar u^{(k)}) + \widehat \nabla \hat \eta_\theta(\bar u^{(k)})\|_2^2 \leq \frac{2L^2 (\hat f_\theta(\bar u^{(0)}) - \hat f_\theta(\bar u^{(K)}))}{K}
	\end{align*}
	where $L = \max_{\gamma \in [\gamma_{\min}, \gamma_{\max}]} \sqrt{\gamma}(L_g + \gamma^{-1} + (\tilde \sigma/\sigma)^{2}\ell_{\traj}^2L_h)$.
	Finally, using Prop.~\ref{prop:trunc_approx}, we get 
	\[
	\min_{k =0,\ldots, K-1}\|\nabla \hat f_\theta(\bar u^{(k)})\|_2 \leq L \sqrt{\frac{ 2(\hat f_\theta(\bar u^{(0)}) - \hat f_\theta(\bar u^{(K)}))}{K}}+ 
\theta \tilde \sigma^2 L_h^2  L_{\tilde x}   \ell_{\traj} M_{\traj}^2+ \theta^2 \tilde \sigma^4 L_h^3L_{\traj}\ell_{\traj}^3 M_{\traj}^2 + \tau p \tilde \sigma^2 L_h L_{\traj}  \ell_{\traj}.	\]
\end{proof}
The following proposition ensures that on any compact set there exists a step-size such that this criterion is satisfied.
\begin{restatable}{proposition}{quadcompactbound}\label{prop:compact_bound}
	Under Asm.~\ref{ass:cvg}, for any compact set $C$ there exists $M_C>0$ such that for any $\bar u\in C, \bar v \in C$, the model $m_{f_\theta}$ approximates the surrogate risk-sensitive cost as
	\[
	|\hat f_{\theta}(\bar u + \bar v) - m_{f_\theta}(\bar u + \bar v; \bar u) | \leq \frac{M_C\|\bar v\|^2}{2}.\]
\end{restatable}
\begin{proof}
	Denote $R_C =\max_{\bar u \in C} \|\bar u\|_2$.
	Denote   $X = \nabla \traj(\bar u)$, $H = \nabla^2 h (\bar x) $.
	Following proof of Prop.~\ref{prop:risk_approx_comput}, we have
	\begin{align*}
		m_{f_\theta}(\bar u + \bar v; \bar u)  = &  h(\traj(\bar u) + \nabla\traj(\bar u)^\top\bar v) 
		- \frac{1}{2\theta}\log \det(\id - \theta\sigma^2 X HX^\top) \\
		& + \frac{\theta\sigma^2}{2}\nabla h( \traj(\bar u) + \nabla\traj(\bar u)^\top\bar v)^\top X^\top(\id_{\horizon p} - \theta\sigma^2 X HX^\top)^{-1} X \nabla h( \traj(\bar u) + \nabla\traj(\bar u)^\top\bar v) \\
		& + g(\bar u + \bar v)
	\end{align*}
In the following denote $ \mathring h = \nabla h( \traj(\bar u) + \nabla\traj(\bar u)^\top\bar v)$.
	On the other side, denote $\bar y = \traj(\bar u+ \bar v)$, $Y = \nabla \traj(\bar u + \bar v) $ and $\hat h = \nabla h(\traj(\bar u + \bar v)) = \nabla h(\bar y)$, such that 
	\begin{align*}
		\hat f_\theta(\bar u+\bar v) = h(\bar y) -\frac{1}{2\theta}\log \det(\id -\theta\sigma^2 YHY^\top) +\frac{\theta\sigma^2}{2}\hat h^\top Y^\top (\id - \theta\sigma^2 YHY^\top)^{-1}Y\hat h + g(\bar u + \bar v)
	\end{align*}
	First we have using $\bar x_* \in \argmin_{\bar x \in \reals^{\horizon d}} h(\bar x)$,
	\begin{align*}
		|h(\traj(\bar u + \bar v)) - h(\traj(\bar u) + \nabla \traj (\bar u)^\top \bar v)| & = |\frac{1}{2}(\traj(\bar u + \bar v) +\traj(\bar u) +\nabla \traj(\bar u)^\top \bar v -2\bar x^*)^\top  H (\traj(\bar u  + \bar v) - \traj(\bar u) - \nabla \traj(\bar u)^\top \bar v)| \\
		& \leq \frac{1}{4}(2M_{\traj} + \ell_{\traj} R_C)L_hL_{\traj}\|\bar v\|_2^2.
	\end{align*}
Then denote 
\[
f(X) = - \frac{1}{2\theta}\log \det(\id - \theta\sigma^2 X HX^\top)
\]
such that 
\[
	\|\nabla f(X)\|_2 = \sigma^2 \|(\id -\theta\sigma^2 XHX^\top)^{-1}XH\|_2 \leq \frac{\sigma^2L_h\ell_{\traj}}{1-\theta\sigma^2 L_h\ell_{\traj}^2}.
\]
Therefore
\begin{align*}
|f(X) - f(Y)| & \leq \ell_f \|\nabla \traj(\bar u+ \bar v) - \nabla \traj(\bar u)\|_2  \\
&  \leq \frac{L_h\ell_{\traj}L_{\traj}}{1-\theta\sigma^2L_h\ell_{\traj}^2}\|\bar v\|_2
\end{align*}
where $\ell_f$ is the Lipschitz continuity of $f$ for $X$ s.t. $\|X\|_2 \leq \ell_{\traj}$.

Now for the last term, we have
\begin{align*}
	\Tr(F(Y)\hat h \hat h^\top) - \Tr(F(X)\mathring h \mathring h^\top) = 
	\Tr((F(Y) -F(X))\hat h \hat h^\top ) + \Tr(F(X) (\hat h \hat h^\top - \mathring h \mathring h^\top))
\end{align*}
where $F(X) = X^\top (\id - \theta\sigma^2 XH X^\top)^{-1}X$.
Define for $M \in \reals^{\horizon d \times \horizon d}$ with $M\succeq 0$,
\[
f_M(X) = \frac{1}{2}\Tr(MX^\top (\id - \theta\sigma^2 XH X^\top)^{-1}X ).
\]
We have
\begin{align*}
\|\nabla f_M(X)\|_2 = & \| (\id - \theta\sigma^2 X H X^\top)^{-1}XM  + \theta\sigma^2 (\id -\theta\sigma^2 XHX^\top)^{-1} XMX^\top (\id - \theta\sigma^2 XHX^\top)^{-1}XH\|_2 \\
& \leq \frac{\|M\|_2 \ell_{\traj}}{1- \theta\sigma^2 L_h \ell_{\traj}^2} + \frac{\theta\sigma^2 \|M\|_2\ell_{\traj}^3L_h}{(1- \theta\sigma^2 L_h \ell_{\traj}^2)^2}.
\end{align*}
Therefore 
\begin{align*}
	|\Tr((F(Y) -F(X))\hat h \hat h^\top )| &  \leq \ell_{f_{\hat h \hat h^\top}}\|Y-X\|_2\\
	& \leq \ell_{h, \traj}^2\left(\frac{ \ell_{\traj}}{1- \theta\sigma^2 L_h \ell_{\traj}^2} + \frac{\theta\sigma^2 \ell_{\traj}^3L_h}{(1- \theta\sigma^2 L_h \ell_{\traj}^2)^2}\right)L_{\traj}\|\bar v\|_2,
\end{align*}
where $\ell_{f_{\hat h \hat h^\top}}$ is the Lipschitz continuity of $f_{\hat h \hat h^\top}$ for $X$ s.t. $\|X\|_2 \leq \ell_{\traj}$.
Finally, 
\begin{align*}
|\Tr(F(X) (\hat h \hat h^\top - \mathring h \mathring h^\top))| & = |\Tr(\hat h + \mathring h)^\top F(X)(\hat h - \mathring h)| \\
& \leq (2\ell_{h, \traj} + L_h \ell_{\traj} R_C)\frac{\ell_{\traj}^2}{1- \theta\sigma^2 L_h \ell_{\traj}^2}L_hL_{\traj} \frac{\|\bar v\|_2^2}{2}.
\end{align*}

Combining all terms we get
\begin{align*}
|\hat f_\theta(\bar u+\bar v) - m_{f_\theta}(\bar u + \bar v)| \leq & \frac{1}{2}(2M_{\traj} + \ell_{\traj} R_C)L_hL_{\traj}\frac{\|\bar v\|_2^2}{2} \\
& + \frac{2L_h\ell_{\traj}L_{\traj}}{(1-\theta\sigma^2 L_h\ell_{\traj}^2)R_C}\frac{\|\bar v\|_2^2}{2} \\
&+ \theta\sigma^2 \ell_{h, \traj}^2\left(\frac{ \ell_{\traj}}{1- \theta\sigma^2 L_h \ell_{\traj}^2} + \frac{\theta\sigma^2 \ell_{\traj}^3L_h}{(1- \theta\sigma^2 L_h \ell_{\traj}^2)^2}\right)L_{\traj}\frac{\|\bar v\|_2^2}{2} \\
& + \frac{\theta\sigma^2}{2}(2\ell_{h, \traj} + L_h \ell_{\traj} R_C)\frac{\ell_{\traj}^2}{1- \theta\sigma^2 L_h \ell_{\traj}^2}L_hL_{\traj} \frac{\|\bar v\|_2^2}{2}
\end{align*}
This concludes the proof with

\begin{align*}
	M_C = & \frac{1}{2}(2M_{\traj} + \ell_{\traj} R_C)L_hL_{\traj} + \frac{2\sigma^2L_h\ell_{\traj}L_{\traj}}{(1-\theta\sigma^2 L_h\ell_{\traj}^2)R_C} \\
	& +
	\theta\sigma^2 \ell_{h, \traj}^2\left(\frac{ \ell_{\traj}}{1- \theta\sigma^2 L_h \ell_{\traj}^2} + \frac{\theta\sigma^2 \ell_{\traj}^3L_h}{(1- \theta\sigma^2 L_h \ell_{\traj}^2)^2}\right)L_{\traj}
	+ \frac{\theta\sigma^2}{2}(2\ell_{h, \traj} + L_h \ell_{\traj} R_C)\frac{\ell_{\traj}^2}{1- \theta\sigma^2 L_h \ell_{\traj}^2}L_hL_{\traj}.
\end{align*}
\end{proof}

Finally the iterates can be forced to stay in a compact set such that the overall convergence is ensured as shown in the following proposition.
\begin{proposition}\label{prop:min_step}
	Let $S_0 = \{\bar u : \hat f_\theta(\bar u) \leq \hat f_\theta(\bar u^{(0)})\}$ be the initial sub-level set of $\hat f_\theta$ and assume $S_0$ is compact.
	Consider the iterations of RegILEQG in~\eqref{eq:min_model_step} under Asm.~\ref{ass:cvg}. Assume that 
	\[
	\gamma_k = \hat \gamma = \min\{\ell_0^{-1}, M^{-1}_C \},
	\]
	where $M_C$ is defined in Prop.~\ref{prop:compact_bound}, and denoting  $\mathcal{B}_{2,1} $ the Euclidean ball of radius 1 centered at 0,
	\begin{align*}
	\ell_0  = \max_{\bar u \in S_0}\|\nabla g(\bar u) + \widehat \nabla \hat \eta_\theta(\bar u)\|_2, \qquad
	C = S_0 + \mathcal{B}_{2,1}.
	\end{align*}
	Then the sufficient decrease condition~\eqref{eq:suff_decrease} is satisfied for all $k$.
\end{proposition}
\begin{proof}
	Given $\bar u^{(k)} \in S_0$, we have from Proposition~\ref{prop:ILEQG_trunc_grad}, using $\gamma_k \leq \ell_0^{-1}$
	\[
	\|\bar u^{(k+1)} - \bar u^{(k)}\|_2 \leq \gamma_k\|\nabla g(\bar u^{(k)}) + \widehat \nabla \hat \eta_\theta(\bar u^{(k)})\|_2 \leq 1.
	\]
	Therefore $\bar u^{(k+1)} \in S_0 + \mathcal{B}_{2,1} = C$ and $\bar u^{(k)} \in C$. They satisfy then, using $\gamma_k \leq M^{-1}_C$,
	\[
	\hat f_\theta(\bar u^{(k+1)}) \leq m_{f_\theta}(\bar u^{(k+1)}; \bar u^{(k)}) + \frac{M_C}{2}\|\bar u^{(k+1)} - \bar u^{(k)}\|_2^2 \leq m_{f_\theta}(\bar u^{(k+1)}; \bar u^{(k)}) + \frac{1}{2\gamma_k}\|\bar u^{(k+1)} - \bar u^{(k)}\|_2^2
	\]
	Therefore $\bar u^{(k+1)} \in S$.
	The claim follows by recursion starting from $\bar u^{(k)} = \bar u^{(0)} \in S_0$.
\end{proof}

%% file: appendix/05_detailed_exp.tex
\subsection{Discretization of the continuous time settings}
The physical systems we consider below are described by continuous time dynamics of the form
\begin{align*}
\ddot z(t) = f(z(t), \dot z (t), u(t))
\end{align*}
where $z(t), \dot z(t), \ddot z(t)$ denote respectively the position, the speed and the acceleration of the system and $u(t)$ is a force applied on the system. 
The state $x(t) = (x_1(t), x_2(t))$ of the system is defined by the position $x_1(t) = z(t)$ and the speed $x_2(t) = \dot z(t)$ and the continuous cost is defined as
\begin{align*}
J(x, u) = \int_0^T h(x(t))dt + \int_0^T g(u(t))dt \quad \mbox{or} \quad J(x, u) = h(x(T)) + \int_0^T g(u(t))dt,
\end{align*}
where $T$ is the time of the movement and $h, g$ are given convex costs. The discretization of the dynamics with a time step $\delta$ starting from a given state $\hat x_0 = (z_0, 0)$ reads then
\begin{equation*}
\begin{split}
x_{1,t+1} & = x_{1,t} + \delta x_{2,t} \\
x_{2,t+1} & = x_{2,t} + \delta f(x_{1,t}, x_{2,t}, u_t) 
\end{split}
\quad \mbox{for $t = 0,\ldots \horizon-1$}
\end{equation*}
where $\horizon = \lceil T/\delta\rceil$ and the discretized cost reads
\begin{align*}
J(\bar x, \bar u ) = \sum_{t=1}^{\horizon}h(x_t) + \sum_{t=0}^{\horizon-1} g(u_t) \quad \mbox{or} \quad J(\bar x, \bar u) = h(x_\horizon) + \sum_{t=0}^{T-1} g(u_t).
\end{align*}

\subsection{Continuous control settings}
The control settings are illustrated in Fig.~\ref{fig:ctrl_settings}.
\begin{figure}[t]
	\begin{center}
		\begin{subfigure}{0.35\textwidth}
			\includegraphics[width=\textwidth]{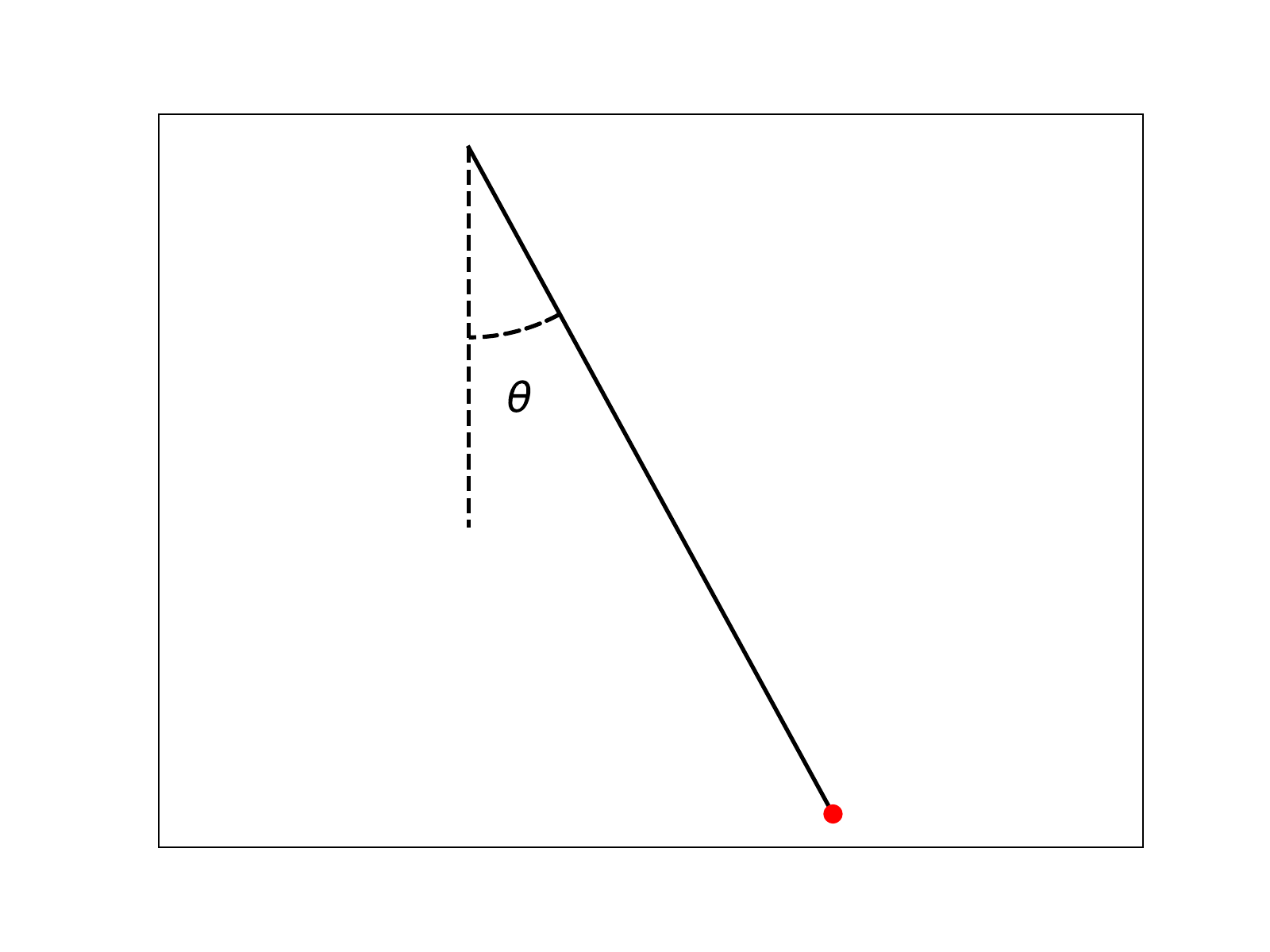}
			\caption{Pendulum.}
		\end{subfigure}
		\begin{subfigure}{0.35\textwidth}
			\includegraphics[width=\textwidth]{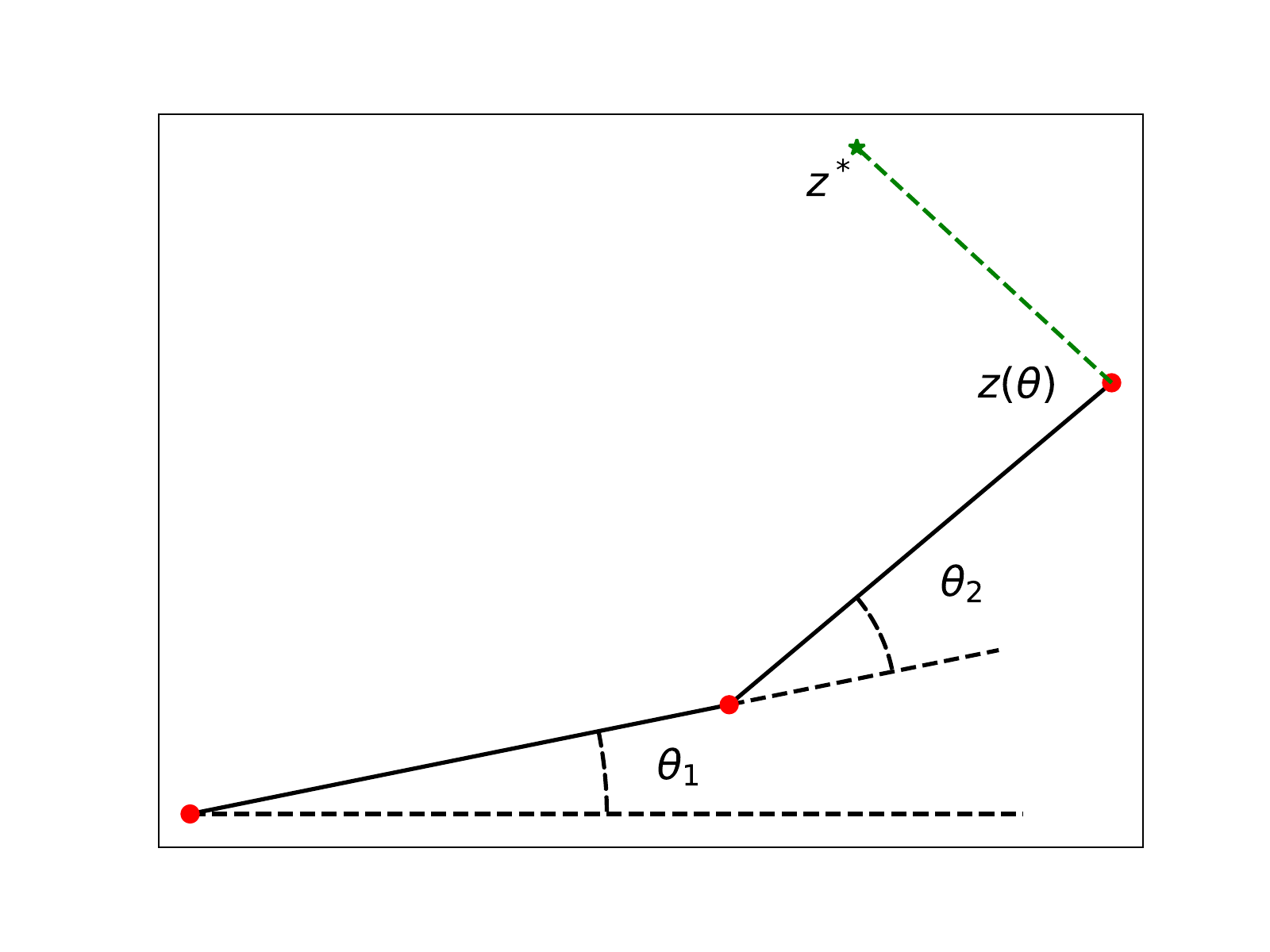}
			\caption{Two-link arm.}
		\end{subfigure}
	\end{center}
	\caption{\label{fig:ctrl_settings} Control settings considered. }
\end{figure}
\paragraph{Pendulum}
We consider a simple pendulum illustrated in Fig.~\ref{fig:ctrl_settings}, where $m=1$ denotes the mass of the bob, $l=1$ denotes the length of the rod, $\theta$ describes the angle subtended by the vertical axis and the rod, and $\mu=0.01$ is the friction coefficient.  Its dynamical evolution reads
\[
\ddot \theta(t) = -\frac{g}{l}\sin \theta(t) - \frac{\mu}{ml^2} \dot\theta(t) + \frac{1}{ml^2}u(t)
\]
The goal is to make the pendulum swing up (i.e. make an angle of $\pi$ radians) and stop at a given time $T$. Formally, the continuous cost reads
\begin{equation}\label{eq:inverse_pendulum_cost}
J(x,u) = 
(\pi -\theta(T))^2 + \lambda_1 \dot \theta(T)^2 + \lambda_2 \int_0^T u^2(t)dt,
\end{equation}
where $x(t) = (\theta(t), \dot \theta(t))$, $\lambda_1> 0$ and $\lambda_2> 0$.

\paragraph{Two-link arm}
We consider the arm model with 2 joints
(shoulder and elbow), moving in the horizontal plane presented by \citep{Li04} and illustrated in Figure~\ref{fig:ctrl_settings}. The dynamics read
\begin{equation}\label{eq:dynamics_2_link_arm_model}
M(\theta(t)) \ddot\theta(t) + C(\theta(t),\dot \theta(t)) + B \dot \theta(t) = u(t),
\end{equation}
where $\theta = (\theta_1, \theta_2)$ is the joint angle vector, $M(\theta) \in \reals^{2\times 2}$ is a positive definite symmetric inertia matrix, $C(\theta, \dot \theta)\in \reals^2$ is a vector centripetal and Coriolis forces, $B \in \reals^{2 \times 2}$ is the joint
friction matrix, and $u \in \reals^2$ is the joint torque controlling the arm. See below for the complete definitions.

The goal is to make the arm reach a feasible target $z^*$ and stop at that point. Denoting $\theta^*(z^*)$ a joint angle pairs that reach the target,
the objective reads then
\begin{equation}\label{eq:2arm_objective}
J(x, u) = \|\theta(T) -\theta^*(z^*)\|_2^2 + \lambda_1 \|\dot \theta(T) \|_2^2 + \lambda_2 \int_0^T \|u(t)\|_2^2dt,
\end{equation} 
where $x(t) = (\theta(t), \dot \theta(t))$, $\lambda_1> 0, \lambda_2>0$.

\paragraph{Detailed two-link arm model}
We detail the the forward dynamics drawn from \eqref{eq:dynamics_2_link_arm_model}. We drop the dependence on $t$ for readability. The dynamics read
\begin{equation*}
\ddot\theta  = M(\theta)^{-1} (u - C(\theta,\dot \theta) - B \dot \theta).
\end{equation*}
The expressions of the different variables and parameters are given by
\begin{align*}
	M(\theta) & = 
	\left( 
	\begin{matrix}
		a_1 + 2a_2\cos \theta_2 & a_3 + a_2 \cos \theta_2 \\
		a_3 + a_2\cos \theta_2 & a_3
	\end{matrix}
	\right) & 
	C(\theta, \dot \theta) & = 
	\left( 
	\begin{matrix}
		- \dot \theta_2(2 \dot \theta_1 + \dot \theta_2) \\
		\dot \theta_1^2
	\end{matrix}
	\right)a_2\sin\theta_2  \\
	B & = 
	\left( 
	\begin{matrix}
		b_{11} & b_{12} \\
		b_{21} & b_{22}
	\end{matrix}
	\right) & &
	\begin{array}{ll}
		a_1 & = k_1 + k_2 + m_2 l_1^2 \\
		a_2 & = m_2l_1d_2\\
		a_3 & = k_2,
	\end{array}
\end{align*}
where $b_{11} = b_{22} = 0.05$, $b_{12} = b_{21} = 0.025$, $l_i$ and $k_i$ are respectively the length (30cm, 33cm) and the moment of inertia (0.025kgm\textsuperscript{2} , 0.045kgm\textsuperscript{2}) of link $i$ , $m_2$ and $d_2$ are respectively the mass (1kg) and the distance (16cm) from the joint center to the center of the mass for the second link. The inverse of the inertia matrix reads\footnote{Note that the dynamics have continuous derivatives if the norm of the denominator is bounded below by a positive constant $0$. We have
	\[
	(a_1 +2a_2\cos(\theta_2))a_3 - (a_3 + a_2 \cos \theta_2)^2 = \alpha - \beta \cos^2 \theta_2
	\]
	with 
	\begin{equation*}
	\alpha = a_3(a_1-a_3) = k_1k_2 +m_2l_1^2k_2  \qquad
	\beta = a_2^2 = m_2^2l_1^2d_2^2,
	\end{equation*}
	which gives $\alpha = 9.1125 \times10^{-2}$ and $\beta = 2.304\times10^{-3}$. Therefore it is bounded below by a positive constant, the function is continuously differentiable.} 
\[
M(\theta)^{-1} = \frac{1}{(a_1 +2a_2\cos(\theta_2))a_3 - (a_3 + a_2 \cos \theta_2)^2}
\left( 
\begin{matrix}
a_3 & -(a_3 + a_2 \cos \theta_2 )\\
-(a_3 + a_2\cos \theta_2) &  a_1 + 2a_2\cos \theta_2
\end{matrix}
\right).
\]

\subsection{Noise modeling details}
Otherwise the modeled noise led experimentally to a chaotic behavior. Precisely we use for the risk-sensitive cost,
\begin{equation*}
\begin{split}
x_{1,t+1} & = x_{1,t} + \delta x_{2,t}  \\
x_{2,t+1} & = x_{2,t} + \delta f(x_{1,t}, x_{2,t}, u_t + w_{t}) 
\end{split}
\quad \mbox{for $t = 0,\ldots ,\horizon-1$},
\end{equation*}  
with $ w_t \sim \mathcal{N}(0,\sigma_0^2\id) $
and for the test cost,
\begin{equation}\nonumber
\begin{split}
x_{1,t+1} & = x_{1,t} + \delta x_{2,t}  \\
x_{2,t+1} & = x_{2,t} + \delta f(x_{1,t}, x_{2,t}, u_t + \rho \mathbb{1}(t = t_w)) 
\end{split}
\quad \mbox{for $t = 0,\ldots ,\horizon-1$},
\end{equation} 
where $\rho \sim \mathcal{N}(0, \sigma_{test}/\sigma_0 \id_p)$ and the plots are shown for increasing $\sigma_{test}$.
For the pendulum problem we used $\sigma_0=1$. For the two-link arm we use $\sigma_0 = 1/\|M(\theta)^{-1}\|$ to normalize the noise in the risk-sensitive and the test costs.
We leave the analysis of the choice of $\sigma$ for future work.

\subsection{Optimization details}
\paragraph{Convergence results}
For Fig.~\ref{fig:conv}, we took $\lambda_1 = 0.1$, $\lambda_2 = 0.01$, $T=5$, in~\eqref{eq:inverse_pendulum_cost} for an horizon $\horizon=100$ and $\theta=4$. We present in Fig.~\ref{fig:conv_two_link} the convergence obtained for the two-link arm problem, where we used the same parameters for $\lambda_1, \lambda_2, T, \horizon, \theta$. The best step-sizes found after the burn-in phase were $8$ for RegILEQG and $0.5$ for ILEQG. Again the advantage of the regularized approach is that it can select bigger step-sizes while staying stable.

\begin{figure}
	\begin{center}
\includegraphics[width=0.32\linewidth]{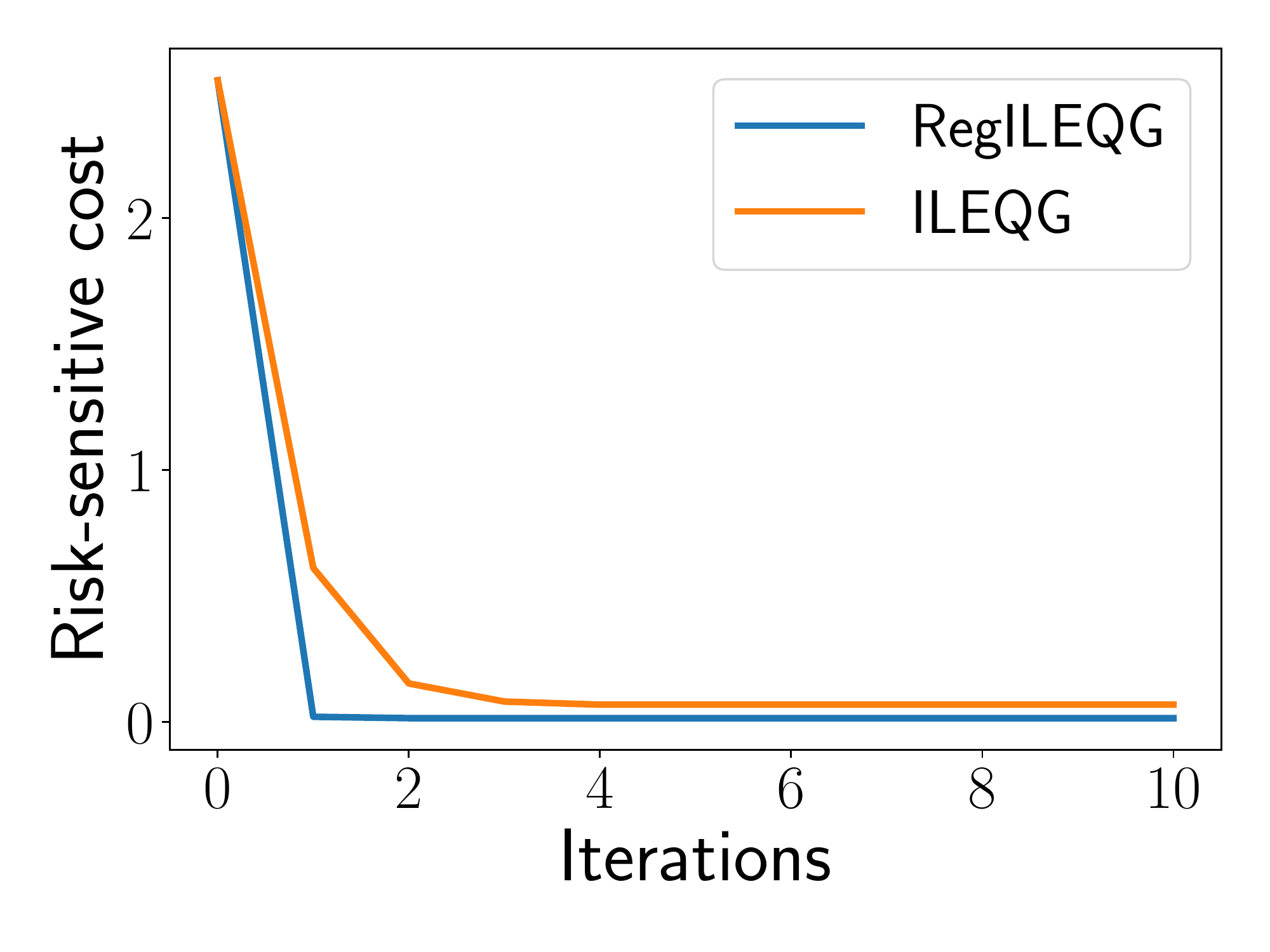}
\includegraphics[width=0.32\linewidth]{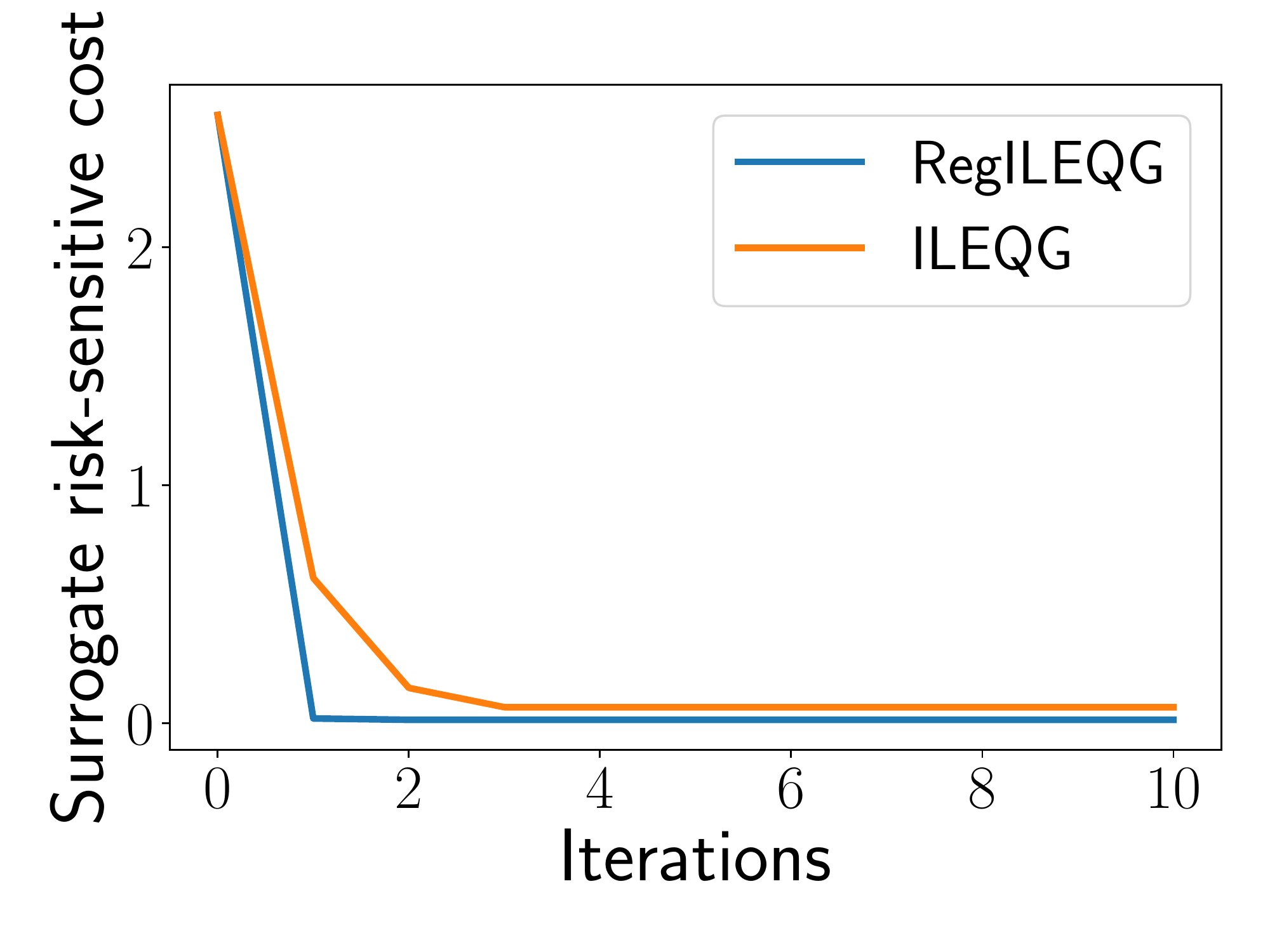}
	\end{center}
	\caption{Convergence of iterative linearized methods, \\
		 RegILEQG and  ILEQG,  on the two-link arm problem. \label{fig:conv_two_link}}
\end{figure}

\paragraph{Robustness results}
For both settings we used RegILEQG with a burn-in phase of 10 iterations and a grid of step-sizes $2^{i}$ for $i \in \{-5,5\}$. We run the algorithm for 50 iterations and take the best solution according to the surrogate risk-sensitive function.

For the pendulum problem we used $\lambda_1=10$, $\lambda_2=10^{-3}$, $T=5$, for an horizon $\horizon=100$. For the two-link arm problem we used $\lambda_1=10^{-2}$ and $\lambda_2=10^{-3}$, $T=5$, and the same horizon.

%% file: risk_ctrl.bbl
\begin{thebibliography}{21}
\providecommand{\natexlab}[1]{#1}
\providecommand{\url}[1]{\texttt{#1}}
\expandafter\ifx\csname urlstyle\endcsname\relax
  \providecommand{\doi}[1]{doi: #1}\else
  \providecommand{\doi}{doi: \begingroup \urlstyle{rm}\Url}\fi

\bibitem[Bellman(1971)]{Bell67}
R.~Bellman.
\newblock \emph{Introduction to the mathematical theory of control processes},
  volume~2.
\newblock Academic press, 1971.

\bibitem[Ben-Tal and Teboulle(1986)]{ben1986expected}
A.~Ben-Tal and M.~Teboulle.
\newblock Expected utility, penalty functions, and duality in stochastic
  nonlinear programming.
\newblock \emph{Management Science}, 32\penalty0 (11):\penalty0 1445--1466,
  1986.

\bibitem[Ben-Tal and Teboulle(2007)]{ben2007old}
A.~Ben-Tal and M.~Teboulle.
\newblock An old-new concept of convex risk measures: The optimized certainty
  equivalent.
\newblock \emph{Mathematical Finance}, 17\penalty0 (3):\penalty0 449--476,
  2007.

\bibitem[Bertsekas(2018)]{Bert18}
D.~P. Bertsekas.
\newblock \emph{Abstract dynamic programming}.
\newblock Athena Scientific, 2nd edition, 2018.

\bibitem[De~O.~Pantoja(1988)]{Panto88}
J.~De~O.~Pantoja.
\newblock Differential dynamic programming and {N}ewton's method.
\newblock \emph{International Journal of Control}, 47\penalty0 (5):\penalty0
  1539--1553, 1988.

\bibitem[Dunn and Bertsekas(1989)]{Dunn89}
J.~C. Dunn and D.~P. Bertsekas.
\newblock Efficient dynamic programming implementations of {N}ewton's method
  for unconstrained optimal control problems.
\newblock \emph{Journal of Optimization Theory and Applications}, 63\penalty0
  (1):\penalty0 23--38, 1989.

\bibitem[Dvijotham et~al.(2014)Dvijotham, Fazel, and
  Todorov]{dvijotham2014universal}
K.~Dvijotham, M.~Fazel, and E.~Todorov.
\newblock Universal convexification via risk-aversion.
\newblock In \emph{Proceedings of the Thirtieth Conference on Uncertainty in
  Artificial Intelligence}, pages 162--171, 2014.

\bibitem[Farshidian and Buchli(2015)]{farshidian2015risk}
F.~Farshidian and J.~Buchli.
\newblock Risk sensitive, nonlinear optimal control: Iterative linear
  exponential-quadratic optimal control with {G}aussian noise.
\newblock \emph{arXiv preprint arXiv:1512.07173}, 2015.

\bibitem[Glover and Doyle(1988)]{glover1988state}
K.~Glover and J.~C. Doyle.
\newblock State-space formulae for all stabilizing controllers that satisfy an
  $h^\infty$-norm bound and relations to relations to risk sensitivity.
\newblock \emph{Systems \& Control Letters}, 11\penalty0 (3):\penalty0
  167--172, 1988.

\bibitem[Hassibi et~al.(1999)Hassibi, Sayed, and
  Kailath]{hassibi1999indefinite}
B.~Hassibi, A.~H. Sayed, and T.~Kailath.
\newblock \emph{Indefinite-Quadratic Estimation and Control: A Unified Approach
  to H2 and H-infinity Theories}, volume~16.
\newblock SIAM, 1999.

\bibitem[Helton and James(1999)]{helton1999extending}
J.~W. Helton and M.~R. James.
\newblock \emph{Extending H-infinity control to nonlinear systems: Control of
  nonlinear systems to achieve performance objectives}, volume~1.
\newblock SIAM, 1999.

\bibitem[Jacobson(1973)]{jacobson1973optimal}
D.~Jacobson.
\newblock Optimal stochastic linear systems with exponential performance
  criteria and their relation to deterministic differential games.
\newblock \emph{IEEE Transactions on Automatic control}, 18\penalty0
  (2):\penalty0 124--131, 1973.

\bibitem[Li and Todorov(2004)]{Li04}
W.~Li and E.~Todorov.
\newblock Iterative linear quadratic regulator design for nonlinear biological
  movement systems.
\newblock In \emph{1st International Conference on Informatics in Control,
  Automation and Robotics}, volume~1, pages 222--229, 2004.

\bibitem[Li and Todorov(2007)]{Li07}
W.~Li and E.~Todorov.
\newblock Iterative linearization methods for approximately optimal control and
  estimation of non-linear stochastic system.
\newblock \emph{International Journal of Control}, 80\penalty0 (9):\penalty0
  1439--1453, 2007.

\bibitem[Nesterov(2013)]{nesterov2013introductory}
Y.~Nesterov.
\newblock \emph{Introductory lectures on convex optimization: A basic course},
  volume~87.
\newblock Springer Science \& Business Media, 2013.

\bibitem[Ponton et~al.(2016)Ponton, Schaal, and Righetti]{ponton2016risk}
B.~Ponton, S.~Schaal, and L.~Righetti.
\newblock On the effects of measurement uncertainty in optimal control of
  contact interactions.
\newblock In \emph{The 12th International Workshop on the Algorithmic
  Foundations of Robotics WAFR}, 2016.

\bibitem[Roulet et~al.(2019)Roulet, Srinivasa, Drusvyatskiy, and
  Harchaoui]{roulet2019iterative}
V.~Roulet, S.~Srinivasa, D.~Drusvyatskiy, and Z.~Harchaoui.
\newblock Iterative linearized control: Stable algorithms and complexity
  guarantees.
\newblock In \emph{Proceedings of the 36th International Conference on Machine
  Learning}, 2019.

\bibitem[Sideris and Bobrow(2005)]{Side05}
A.~Sideris and J.~E. Bobrow.
\newblock An efficient sequential linear quadratic algorithm for solving
  nonlinear optimal control problems.
\newblock In \emph{Proceedings of the American Control Conference}, pages
  2275--2280, 2005.

\bibitem[Sopasakis et~al.(2019)Sopasakis, Herceg, Bemporad, and
  Patrinos]{sopasakis2019risk}
P.~Sopasakis, D.~Herceg, A.~Bemporad, and P.~Patrinos.
\newblock Risk-averse model predictive control.
\newblock \emph{Automatica}, 100:\penalty0 281--288, 2019.

\bibitem[Speyer et~al.(1974)Speyer, Deyst, and
  Jacobson]{speyer1974optimization}
J.~Speyer, J.~Deyst, and D.~Jacobson.
\newblock Optimization of stochastic linear systems with additive measurement
  and process noise using exponential performance criteria.
\newblock \emph{IEEE Transactions on Automatic Control}, 19\penalty0
  (4):\penalty0 358--366, 1974.

\bibitem[Whittle(1981)]{whittle1981risk}
P.~Whittle.
\newblock Risk-sensitive linear/quadratic/{G}aussian control.
\newblock \emph{Advances in Applied Probability}, 13\penalty0 (4):\penalty0
  764--777, 1981.

\end{thebibliography}
